\documentclass[english, 11pt]{amsart} 

\usepackage{cite}
\usepackage{amsmath}
\usepackage{amsfonts}
\usepackage{amssymb}
\usepackage{ae}
\usepackage{mathrsfs,array,graphicx}
\usepackage[all,ps]{xy}
\usepackage{hyperref}

\usepackage{note}
\usepackage{a4wide}
\usepackage{std_macros}
\newcommand{\Pol}{\textup{Pol}}
\newcommand{\Ram}{\textup{Ram}^+}
\newcommand{\Unr}{\textup{Unr}^+}
\renewcommand{\G}{\mathbb{G}} 
\renewcommand{\Gm}{\mathbb{G}_\textup{m}}
\renewcommand{\D}{\mathbb{D}}

\title{The basic stratum of some simple Shimura varieties}

\author{Arno Kret}

\begin{document}

\begin{abstract}
Under simplifying hypotheses we prove a relation between the $\ell$-adic cohomology of the basic stratum of a Shimura variety of PEL-type modulo a prime of good reduction of the reflex field and the cohomology of the complex  Shimura variety. In particular we obtain explicit formulas for the number of points in the basic stratum over finite fields. We obtain our results using the trace formula and truncation of the formula of Kottwitz for the number of points on a Shimura variety over a finite fields. 
\end{abstract}

\maketitle

\section*{Introduction} 

Let $S$ be a Shimura variety of PEL-type. These Shimura varieties are moduli spaces of Abelian varieties with additional PEL-type structures. Consider the reduction modulo $p$ of $S$ and, for each point $x$ in $S$ modulo $p$ consider the rational Dieundonn\' module $\D(\cA_x)_\Q$ of the Abelian variety $\cA_x$ corresponding to $x$. These Dieundonn\'e modules are isocrystals, and we are interested in them up to isomorphism. When $x$ ranges over the points of $S$ modulo $p$, the isomorphism class of $\D(\cA_x)_\Q$ ranges over a certain \emph{finite} set $B(G_{\qp}, \mu)$ of possibilities. For $b$ in $B(G_{\qp})$, let $S_b \subset S$ be the closed subvariety consisting precisely of those Abelian varieties whose isocrystal is equal to $b$. The collection $\{S_b\}_{b \in B(G_{\gq}, \mu)}$ of subvarieties of $S$ is called the \emph{Newton stratification} of $S$. We study the action of the Frobenius operator on the $\ell$-adic cohomology of the varieties $S_b$, and, additionally, the action of the Hecke algebra of $G$ on these spaces. 
 
In this article we consider a restricted class of certain simple PEL type Shimura varieties, and take $b$ to be the \emph{basic} element of $B(G_{\qp}, \mu)$. We establish a relation between the cohomology of the corresponding \emph{basic stratum} of the Shimura variety $S$ modulo $p$ and the space of automorphic forms on $G$. We show that the space of automorphic forms completely describes the cohomology of the basic stratum as Hecke module, as well as the action Frobenius element.

Let us now give the precise statement of the main result. Let $D$ be a division algebra over $\Q$ equipped with an anti-involution $*$. Write $F$ for the center of $D$. We assume that $F$ is a CM field and that $*$ induces complex conjugation on $F$ and that $D \neq F$. We assume that $F$ is a compositum of a quadratic imaginary extension $\cK$ of $\Q$ and a totally real subfield $F^+$ of $F$. We pick a morphism $h_0$ of $\R$-algebras from $\C$ to $D_\R$ such that $h_0(z)^* = h_0(\li z)$ for all complex numbers $z$, and we assume that the involution $x \mapsto h_0(i)^{-1} x^* h_0(i)$ on $D_\R$ is positive. Then $(D, h)$ induces a Shimura datum $(G, X, h^{-1})$.  Let $K \subset G(\Af)$ be a compact open subgroup of $G$ and let $p$ be a prime number such that $\Sh_K$ has good reduction at $p$ (in the sense of \cite[\S 6]{MR1124982}) and such that $K$ decomposes into a product $K_p K^p$ where $K_p \subset G(\qp)$ is hyperspecial and $K^p$ is sufficiently small so that $\Sh_K$ is smooth. Let $\Sh_K$ be the Shimura variety which represents the corresponding moduli problem of Abelian varieties of PEL-type as defined in \cite{MR1124982}. We write $\cA(G)$ for the space of automorphic forms on $G$. Let $\xi$ be a irreducible complex algebraic representation of $G(\C)$, and let $f_\infty$ be the Clozel-Delorme function at infinity corresponding to the representation $\xi$. We assume that the prime number $p$ splits in the field $\cK$. Let $B$ be the basic stratum of the reduction of the variety $\Sh_K$ modulo a prime $\p$ of the reflex field lying above $p$, and let $\fq$ be the residue field of $E$ at $\p$. We write $\Phi_\p \in \Gal(\lfq/\fq)$ for the geometric Frobenius $x \mapsto x^{q^{-1}}$. Let $\cL$ be the restriction to $B_{\lfq, \et}$ of the $\ell$-adic local system associated to $\xi$ on $\Sh_{K, \lfq, \et}$ \cite[\S 6]{MR1124982}. Let $f^p$ be a $K^p$-spherical Hecke operator in $\cH(G(\Af^p))$, where $\Af^p$ is the ring of finite ad\`eles with trivial component at $p$. Finally we work under a technical simplifying condition on the geometry of the variety $B$. 
Let $b \in B(G_{\qp}, \mu)$ be the isocrystal with additional $G$-structure corresponding to the basic stratum. The group $G(\Qp)$ is equal to $\qp^\times \times \Gl_n(F^+ \otimes \Qp)$ and the set $B(G_{\qp})$ decomposes along the irreducible factors of the algebra $F^+ \otimes \Qp$. Therefore we have for each $F^+$-place $\wp$ above $p$ an isocrystal $b_\wp \in B(\Gl_n(F_\wp^+))$. The condition on $b$ is that the only slope in $b_{\wp}$ which could occur with multiplicity $>1$ is the slope $0$, \ie we assume that for each $F^+$-place $\wp$ that if  $b_\wp$ is \'etale or simple as object in the category of isocrystals with additional $\Gl_n(F_\wp^+)$-structure. Under these conditions we have the following theorem 

\begin{theorem}[see Theorem \ref{maintheorem}]
The trace of the correspondence $f^p \times \Phi_\p^\alpha$ acting on the alternating sum of the cohomology spaces
$\sum_{i=0}^\infty (-1)^i \uH_{\et}^i(B_{\lfq}, \iota^* \cL)$ is equal to 
\begin{equation}\label{introBB}
P(q^\alpha) \lhk \sum_{\bo{\pi \subset \cA(G)}{\dim(\pi) = 1,\pi_p=\textup{unr.}}}  \zeta_{\pi}^{\alpha} \cdot \Tr(f^p, \pi^p) + \eps \sum_{\bo{\pi \subset \cA(G)}{\pi_p=\textup{ St. type}}} \zeta_{\pi}^{\alpha} \cdot \Tr(f^p, \pi^p)\rhk. 
\end{equation}
for all positive integers $\alpha$. The condition that $\pi_p$ is of Steinberg type means that for every $F^+$-place $\wp$ above $p$ that 
\begin{enumerate}
\item if the component of the basic isocrystal at $\wp$ is not \'etale (\ie has non-zero slopes),  then $\pi_\wp$ is a twist by an unramified character  of the Steinberg representation of $\Gl_n(F^+_\wp)$; 
\item if the component at $\wp$ is \'etale (all slopes are $0$), then the representation $\pi_\wp$ is unramified and generic. 
\end{enumerate}
The symbol $\eps \in \Z^\times$  in Equation \eqref{introBB} is a sign equal to $(-1)^{(n-1)\#\Ram_p}$ where $\Ram_p$ is the set of $F^+$-places $\wp$ dividing $p$ such that the isocrystal $b_\wp$ is not \'etale. The number $\zeta_{\pi}$ is a certain Weil-$q$-numbers whose weight depends on $\xi_\C$ (see Lemma \ref{weilqnumber} for the precise statement). The symbol $P(q^\alpha)$ is a certain polynomial function, see Definition \ref{polP} and the discussion below this definition.  
\end{theorem}

To give a flavor of the form of the function $P(q^\alpha)$ let us give it in this introduction under two additional simplifying hypotheses. Let $n$ be the positive integer such that $n^2$ is the dimension of the algebra $D$ over the field $F$. By the classification of unitary groups over the real numbers, the group $G(\R)$ induces for each infinite $F^+$-place $v$ a set of non-negative numbers $\{p_v, q_v\}$ with $p_v + q_v =n$. Let us assume in this introduction that $p_v = 0$ for all $v$ except for one infinite $F^+$-place $v_0$. Secondly we assume that $p$ splits completely in the field $F^+$, then there is a polynomial $\Pol \in \C[X]$, such that $P(q^\alpha)$ is equal to $\Pol|_{X = q^\alpha}$. Our condition on the basic isocrystal corresponds to the condition that the number $p_{v_0}$ is coprime to $n$ (see Subsection \S \ref{subsectionisocrystals}). Write $s$ for the signature $p_{v_0}$. The polynomial $P(q^\alpha)$ is equal to the evaluation of the polynomial
\begin{equation}\label{introsimpleformula}
q^{\alpha \frac {s(n-s)}2 } \sum X_{i_1}^\alpha X_{i_2}^\alpha \cdots X_{i_s}^\alpha \in \C[X_1^{\pm 1}, X_2^{\pm 1}, \ldots, X_n^{\pm 1}],
\end{equation}
at the point $X_1 = q^{\frac {1 -n}2}, X_2 = q^{\frac {3 - n}2}, \ldots, X_n = q^{\frac {n -1}2}$. In the sum of Equation \eqref{introsimpleformula} the indices $i_1, i_2, \ldots, i_s$ range over the set $\{1, 2, \ldots, n\}$ and satisfy the conditions
\begin{enumerate}
\item[$\bullet$] $i_1 < i_2 < i_3 < \ldots < i_s$;
\item[$\bullet$] $i_1 = 1$;
\item[$\bullet$] If $s > 1$ there is an additional condition: For each subindex $j \in \{2, \ldots, s\}$ we have $i_j < 1+ \frac ns (j-1)$. 
\end{enumerate}
In the case of Harris and Taylor \cite{MR1876802} the polynomial $\Pol(q^\alpha)$ is equal to $1$ (the basic stratum is then a \emph{finite} variety). 

Let us comment on the strategy of proof of Theorem \ref{maintheorem}. The formula of Kottwitz for Shimura varieties of PEL-type \cite{MR1124982} gives an expression for the number of points over finite fields on these varieties at primes of good reduction. The formula is an essential ingredient for the construction of Galois representations associated to automorphic forms: The formula allows one to compare the trace of the Frobenius operator with the trace of a Hecke operator acting on the cohomology of the variety. In this article we truncate the formula of Kottwitz to only contain the conjugacy classes which are \emph{compact at $p$}. Thus we count virtual Abelian varieties with additional PEL-type structure lying in the \emph{basic stratum}. The stabilization argument of Kottwitz carried out in his Ann Arbor article \cite{MR1044820} still applies because the notion of $p$-compactness is stable under stable conjugacy. After stabilizing we obtain a sum of stable orbital integrals on the group $G(\A)$, which can be compared with the geometric side of the trace formula. Ignoring endoscopy and possible non-compactness of the variety, the geometric side is equal to the \emph{compact} trace $\Tr(\chi_c^{G(\qp)} f, \cA(G))$ as considered by Clozel in his article on the fundamental lemma \cite{MR1068388}. Using base change and Jacquet-Langlands we compare this compact trace with the twisted trace of a certain truncated Hecke operator acting on automorphic representations of the general linear group. We arrive at a local combinatorial problem at $p$ to classify the representations which contribute (rigid representations, Section \ref{sectionrigid}), and the computation of the compact trace of the Kottwitz function on these representations (Section \ref{sectioncombinatorics}), which is a problem that we solve only partially in this article (hence the conditions on the basic isocrystal). 

One expects that a formula such as in Equation \eqref{introsimpleformula} can also be obtained for other Shimura varieties associated to unitary groups. However, such a formula will certainly be more complicated: We already mentioned that the above statements concerning the trace formula are too naive for more general unitary Shimura varieties. If the group $G$ has endoscopy, then the spectral side $\Tr(\chi_c^{G(\qp)} f, \cA(G))$ has to change into a sum of truncated transferred functions acting on automorphic representations of endoscopic groups (cf. \cite{harrisproject} \cite{MR2800722}). If furthermore the Shimura variety is non-compact, then additional problems occur. In the formula of Kottwitz for the number of points on a Shimura variety another sum should be included, indexed by the parabolic subgroups of $G$, and various \'etale cohomology spaces should be changed to intersection cohomology (cf. \cite{MR2567740}).  Leaving these complications for future work, we assume that the Shimura variety is such that the problems do not occur: We work in  the setting of Kottwitz in his Inventiones article \cite{MR818353} where the Shimura variety is associated to a division algebra. In this case the Shimura varieties are compact and there is essentially no endoscopy, hence the truth of the naive formula and hence the word ``simple'' in the title of this article. 

The above formula is established in Section \ref{establishmainformula}. In Subsections \ref{application.points} and \ref{application.dimension} we use the formula to deduce two applications, in the first we express the zeta function of the basic stratum in terms of automorphic data, in the second application we derive a dimension formula for the basic stratum. In the first Section \S 1 we carry out the necessary local computations at $p$ to derive the explicit formula for the polynomial $P(q^\alpha)$ in Equation \eqref{introsimpleformula}. In Section \S 1 we also prove a vanishing result of the truncated constant terms of the Kottwitz function due to the imposed conditions on the basic isocrystal (Proposition \ref{chicfp}). This result is the technical reason for the simplicity of the formula in Theorem \ref{maintheorem}: without the conditions on the basic isocrystal, Equation \eqref{introsimpleformula} is more complicated and involves a larger class of representations at $p$. In Section \S 2 we apply the  Moeglin-Waldspurger Theorem to determine which smooth irreducible representations of the general linear group can occur as components of discrete automorphic representations at finite places of a number field. This result is important for the final argument in Section 3.

We end this introduction with a discussion on upcoming work. We are working on the generalization of our main Theorem \ref{maintheorem} to more general Shimura varieties of PEL type and of type (A). Furthermore the same strategy can be used to study the basic stratum of the variety at primes where the group $G_{\qp}$ is unramified, but the prime $p$ is inert (instead of split) in the quadratic imaginary subfield of $F$. At least at the split primes the non-basic Newton strata $b \in B(G_{\qp}, \mu)$ may be attacked as well using this method. However, explicit results require computation of the Jacquet modules $(\pi_p)_N$ where $P=MN$ is the standard parabolic subgroup of $G(\qp)$ corresponding to the centralizer $M_b$ in $G_{\qp}$ of the slope morphism $\nu_b$ of $b$. These Jacquet modules have not been computed yet in the current literature.

Theorem \ref{maintheorem} has a simplifying condition on the isocrystal in the basic stratum. This is precisely the case where the only representations at $p$ that contribute to the cohomology are essentially\footnote{See Definition \ref{steinbergtypedef} for the precise statement.} the Steinberg and the trivial representation. The argument in Section 3 gives also a result without the condition on the basic isocrystal. In that case  a large class of ``rigid'' representations $\pi_p$ contribute to the basic stratum, and it is not hard to write down precisely which of them contribute (see \S \ref{sectionrigid}). However, a new challenge is to compute the analogue of the function $P(q^\alpha)$ considered above. This computation seems difficult to us and we have not yet obtained a satisfactory formula, but we do have an unsatisfactory formula which could for example be implemented in a (slow) computer program to yield expressions similar to Theorem \ref{maintheorem} for the number of points in more strata. 

We hope to return to the above ideas in forthcoming work. 

\bigskip

\textbf{Acknowledgments}. I thank my thesis advisor Laurent Clozel for all his help and providing me with the idea of truncating the formula of Kottwitz to the basic stratum, combined with his proposition on compact traces and then computing the result using the trace formula. I also thank my second thesis advisor Laurent Fargues for answering some of my geometric questions on Shimura varieties. Furthermore I thank Guy Henniart, Ioan Badulescu, Marko Tadic,  Chunh-Hui Wang for their help with representation theory, Paul-James White for his help on the trace formula,  Gerard Laumon for a discussion on compact traces, and finally Henniart, Clozel and the referee for their corrections and suggestions.

\tableofcontents

\section{Local computations}\label{sectioncombinatorics}

In this section we compute the compact traces of the functions of Kottwitz against the representations of the general linear group that occur in the cohomology of unitary Shimura varieties. 

\subsection{Notations} Let $p$ be a prime number and let $F$ be a non-Archimedean local field with residue characteristic equal to $p$. Let $\cO_F$ be the ring of integers of $F$, let $\varpi_F \in \cO_F$ be a prime element. We write $\Fq$ for the residue field of $\cO_F$, and the number $q$ is by definition its cardinal. The symbol $G_n$ denotes the locally compact group $\Gl_n(F)$. If confusion is not possible then we drop the index $n$ from the notation. We call a parabolic subgroup $P$ of $G$ \emph{standard} if it is upper triangular, and we often write $P = MN$ for its standard Levi decomposition. We write $K$ for the hyperspecial subgroup $\Gl_n(\cO_F) \subset G$. Let $\cH(G)$ be the Hecke algebra of locally constant compactly supported complex valued functions on $G$, where the product on this algebra is the one defined by the convolution integral with respect to the Haar measure which gives the group $K$ measure $1$. We write $\cH_0(G)$ for the spherical Hecke algebra of $G$ with respect to $K$. Let $P_0$ be the standard Borel subgroup of $G$, let $T$ be the diagonal torus of $G$, and let $N_0$ be the group of upper triangular unipotent matrices in $G$. 

We write $\one_{G_n}$ for the trivial representation and $\St_{G_n}$ for the Steinberg representation of $G_n$. If $P = MN \subset G$ is a standard parabolic subgroup, then $\delta_P$ is equal to $|\det(m, \iin)|$, where $\iin$ is the Lie algebra of $N$. The induction $\Ind_P^G$ is \emph{unitary} parabolic induction. The Jacquet module $\pi_N$ of a smooth representation is not normalized by convention, for us it is the space of coinvariants for the unipotent subgroup $N \subset G$. For the definition of the constant terms $f^{(P)}$ and the Satake transform we refer to the article of Kottwitz \cite[\S 5]{MR564478}. The valuation $v$ on $F$ is normalized so that $p$ has valuation $1$ and the absolute value is normalized so that $p$ has absolute value $q^{-1}$. Finally, let $x \in \R$ be a real number, then $\lfloor x \rfloor$ (\emph{floor function}) (resp. $\lceil x \rceil$, \emph{ceiling function}) denotes the unique integer in the real interval $(x-1, x]$ (resp. $[x, x+1)$). 

Let $n \in \Z_{\geq 0}$ be a non-negative integer. A \emph{composition} of $n$ is an element $(n_a) \in \Z^k_{\geq 1}$ for some $k \in \Z_{\geq 1}$ such that $n = \sum_{a = 1}^k n_a$. We write $\ell(n_a)$ for $k$ and call it the \emph{length} of the composition. The set of compositions $(n_a)$ of $n$ is in bijection with the set of standard parabolic subgroups of $G_n = \Gl_n(F)$. Under this bijection a composition $(n_a)$ of $n$ corresponds to the standard parabolic subgroup
$$
P(n_a) \  \bydef\  \lbr \left. \vierkantnegen {g_1}\ *\ \ddots\ 0\ {g_k} \in G_n \right| g_a \in G_{n_a}  \rbr  \subset G_n. 
$$
We also consider extended compositions. Let $k$ be a non-negative integer. An \emph{extended composition of $n$} of length $\ell(n_a) = k$ is an element $(n_a) \in \Z^k_{\geq 0}$ such that $n = \sum_{a=1}^k n_a$.

\subsection{Compact traces}\label{subsection.compacttraces} In this subsection we work in a slightly more general setting. We assume that $G$ is the set of $F$-points of a smooth reductive group $\underline G$ over $\cO_F$. We pick a minimal parabolic subgroup $P_0$ of $G$ and we standardize the parabolic subgroups of $G$ with respect to $P_0$. A semisimple element $g$ of $G$ is called \emph{compact} if for some (any) maximal torus $T$ in $G$ containing $g$ the absolute value $|\alpha(g)|$ is equal to $1$ for all roots $\alpha$ of $T$ in $\ig$. We now wish to define compactness for the non semisimple elements $g \in G$. We first pass to the algebraic closure: An element $g \in \underline G(\li F)$ is \emph{compact} if it is semisimple part is compact. A rational element $g \in G$ is \emph{compact} if it is compact when viewed as an element of $\underline G(\li F)$. Let $\chi_c^G$ be the characteristic function on $G$ of the set of compact elements $G_c \subset G$. The subset $G_c \subset G$ is open, closed and stable under stable conjugation. We wish to make the following remark: Let $M$ be a Levi subgroup of $G$ and let $g$ be an element of $M \subset G$. The condition ``$g$ is compact for the group $M$'' is not equivalent to ``$g$ is compact for the group $G$''. We need the two notions and therefore we put the group $G$ in the exponent $\chi_c^G$ to clearly distinguish between the two.

Let $f$ be a locally constant, compactly supported function on $G$. The \emph{compact trace} of $f$ on the representation $\pi$ is defined by $\Tr(\chi_c^G f, \pi)$ where $\chi_c^G f$ is the \emph{point-wise product}. We define $\li f$ to be the conjugation average of $f$ under the maximal compact subgroup $K$ of $G$. More precisely, for all elements $g$ in $G$ the value $\li f(g)$ is equal to the integral $\int_K f(k g k^{-1}) \dd k$ where the Haar measure is normalized so that $K$ has volume $1$. 

Let $P$ be a standard parabolic subgroup of $G$ and let $A_P$ be the split center of $P$, we write $\eps_P = (-1)^{\dim(A_P/A_G)}$. Define $\ia_P$ to be $X_*(A_P)_\R$ and define $\ia_P^G$ to be the quotient of $\ia_P$ by $\ia_G$. To the parabolic subgroup $P$ we associate the subset $\Delta_P \subset \Delta$ consisting of those roots which act non trivially on $A_P$. We write $\ia_0 = \ia_{P_0}$ and $\ia_0^G = \ia_{P_0}^G$. 
For each root $\alpha$ in $\Delta$ we have a coroot $\alpha^\vee$ in $\ia_0^G$. For $\alpha \in \Delta_P \subset \Delta$ we send the coroot $\alpha^\vee \in \ia_0^G$ to the space $\ia_P^G$ via the canonical surjection $\ia_0^G \surjects \ia_P^G$. The set of these restricted coroots $\alpha^\vee|_{\ia_P^G}$ with $\alpha$ ranging over $\Delta_P$ form a basis of the vector space $\ia_P$. By definition the set of fundamental roots $\{ \varpi_\alpha  \in \ia_P^{G*}\,\,|\,\, \alpha \in \Delta_P\}$ is the basis of $\ia_P^{G*} = \Hom(\ia_P^G , \R)$ dual to the basis $\{\alpha^\vee_{\ia_P^G}\}$ of coroots. We let $\tau_P^G$  be the characteristic function on the space $\ia_P^G$ of the \emph{acute Weyl chamber}, 
\begin{equation}\label{acuteweylchamberdef}
\ia_P^{G+} = \lbr x \in \ia_P^G \,\, | \,\, \forall \alpha \in \Delta_P \ \langle \alpha,x \rangle > 0 \rbr. 
\end{equation}
We let $\widehat \tau_P^G$  be the characteristic function on $\ia_P^G$ of the \emph{obtuse Weyl chamber}, 
\begin{equation}\label{obtuseweylchamberdef}
{}^+\ia_P^G = \lbr x \in \ia_P^G \,\, |\,\, \forall \alpha \in \Delta_P \ \langle \varpi_\alpha^G, x \rangle > 0 \rbr.
\end{equation}

Let $P = MN$ be a standard parabolic subgroup of $G$. Let $X(M)$ be the group of rational characters of $M$. The \emph{Harish-Chandra mapping}\footnote{In the definition of the Harish-Chandra map there are different sign conventions possible. For example \cite{MR518111} and \cite{MR0340486} use the convention $q^{ \langle \chi, H_M(m) \rangle} = |\chi(m)|_p$ instead. Our sign follows that of \cite{MR1989693}. In the article \cite{MR1068388} there is no definition of the Harish-Chandra map but we have checked that Clozel uses our normalization.} $H_M$ of $M$ is the unique map from $M$ to $\Hom_\Z(X(M), \R) = \ia_P$, such that the $q$-power $q^{- \langle \chi, H_M(m) \rangle}$ is equal to $|\chi(m)|_p$ for all elements $m$ of $M$ and rational characters $\chi$ in $X(M)$. We define the function $\chi_N$ to be the composition $\tau_P^G \circ (\ia_P \surjects \ia_P^G) \circ H_M$, and we define the function $\widehat \chi_N$ to be the composition $\widehat \tau_P^G \circ (\ia_P \surjects \ia_P^G) \circ H_M$. The functions $\chi_N$ and $\widehat \chi_N$ are locally constant and $K_M$-invariant, where $K_M = \underline M(\cO_F)$. 

\begin{lemma}\label{conditieherschrijven}
Let $P = MN$ be a standard parabolic subgroup of $G$. 
Let $m$ be an element of $M$, then 
\begin{enumerate}
\item $\chi_N(m)$ is equal to $1$ if and only if for all roots $\alpha$ in the set $\Delta_P$ we have $|\alpha(m)| < 1$;
\item $\hat \chi_N(m)$ is equal to $1$ if and only if for all roots $\alpha$ in the set $\Delta_P$ we have $|\varpi_\alpha(m)| < 1$.
\end{enumerate}
\end{lemma}

\begin{proposition}\label{clozelsformula} Let $\pi$ be an admissible $G$-representation of finite length, and let $f$ be an element of $\cH(G)$. The trace $\Tr(f, \pi)$ of $f$ on the representation $\pi$ is equal to the sum $\sum_{P = MN} \Tr_{M, c}\lhk\chi_N \li f^{(P)},  \pi_N(\delta_P^{-1/2})\rhk$ where $P$ ranges over the standard parabolic subgroups of $G$. 
\end{proposition}
\begin{proof}
For the proof see \cite[prop 2.1]{MR986793}. Another proof of this proposition is given in \cite[p. 259--262]{MR1068388}.  
\end{proof}

\begin{proposition}\label{clozelwaldspurger} 
Let $\pi$ be an admissible $G$-representation of finite length, and let $f$ be an element of $\cH(G)$. The compact trace $\Tr(\chi_c^G f, \pi)$ of $f$ on the representation $\pi$ is equal to the sum $\sum_{P = MN} \eps_P \Tr_M\lhk\widehat \chi_N \li f^{(P)},  \pi_N(\delta_P^{-1/2})\rhk$ where $P$ ranges over the standard parabolic subgroups of $G$. 
\end{proposition}
\begin{proof}
This is the Corollary to Proposition 1 in the article \cite{MR1068388}. 
\end{proof}

\begin{remark}
Proposition \ref{clozelsformula} and Proposition \ref{clozelwaldspurger} are true for reductive groups over non-Archimedean local fields in general.
\end{remark}

We record the following corollary. We have a parabolic subgroup $\underline {P_0} \subset \underline G$ such that $P_0 = \underline {P_0}(F)$. Let $I \subset \underline G(\cO_F)$ be the group of elements $g \in \underline G(\cO_F)$ that reduce to an element of the group $\underline {P_0}(\cO_F/\varpi_F)$ modulo $\varpi_F$. The group $I$ is called \emph{the standard Iwahori subgroup of $G$}. A smooth representation $\pi$ of $G$ is called \emph{semi-stable} if it has a non-zero invariant vector under the subgroup $I$ of $G$. 

\begin{corollary}\label{chifpss}
Let $\pi$ be a smooth admissible representation of $G$ such that the trace $\Tr(\chi_c^G  f, \pi)$ does not vanish for some spherical function $f \in \cH_0(G)$. Then $\pi$ is semi-stable.
\end{corollary}
\begin{proof}(cf. \cite[p. 1351--1352]{MR1254737}). By Proposition \ref{clozelwaldspurger} the trace $\Tr(\widehat \chi_N f^{(P)}, \pi_N(\delta_P^{-1/2}))$ is nonzero for some standard parabolic subgroup $P = MN$ of $G$. The function $\widehat \chi_N f^{(P)}$ is $K_M$-spherical, and therefore $\pi_N(\delta_P^{-1/2})$ is an unramified representation of $M$. In particular the representation $\pi_N$ has an invariant vector for the Iwahori subgroup $I$ of $M$. The Proposition 2.4 in \cite{MR571057} gives a linear bijection from the vector space $\lhk \pi_N\rhk^{\underline M(\cO_F)}$ to the vector space $\pi^I$. Therefore the space $\pi^I$ cannot be $0$. 
\end{proof}

\begin{proposition}\label{constantterminduction}
Let $\Omega$ be an open and closed subset of $G$ which is invariant under conjugation by $G$. Let $P = MN$ be a standard parabolic subgroup of $G$. Let $\rho$ be an admissible representation of $M$ of finite length, and let $\pi$ be the induction $\Ind_P^G(\rho)$ of the representation $\rho$ to $G$. Then for all $f$ in $\cH(G)$ the trace 
$\Tr( \chi_\Omega f, \pi)$ is equal to the trace $\Tr( \chi_\Omega (\li f^{(P)}), \rho)$. 
\end{proposition}
\begin{proof}
By the main theorem of \cite{MR338277} we have 
\begin{equation}\label{weylstepA}
\Tr_G(\chi_\Omega f, \pi) = \Tr_M( (\chi_\Omega f)^{(P)}, \rho).
\end{equation}
We prove  that the functions $\chi_\Omega \cdot (\li f^{(P)})$ and $(\chi_\Omega (\li f^{(P)}))$ in $\cH(M)$ have the same orbital integrals. Let $\gamma \in M$. Then the orbital integral $O_\gamma^M( \chi_\Omega \cdot (\li f^{(P)}) )$ equals $O_\gamma^M( \li f^{(P)})$ if $\gamma \in \Omega$ and vanishes for $\gamma \notin \Omega$. 
By Lemma 9 in \cite{MR338277} we have
$$
O_\gamma^M( (\chi_\Omega (\li{f})^{(P)})) = O_\gamma^M( (\li{\chi_\Omega f})^{(P)}) = D(\gamma) O_\gamma^G(\chi_\Omega f),
$$
where 
$D(\gamma) = D_M(\gamma)^{-1/2} D_G(\gamma)^{1/2}$ 
is a certain Jacobian factor for which we do not need to know the definition; we refer to [\textit{loc. cit}] for the definition.
By applying Lemma 9 of [\textit{loc. cit}] once more the orbital integral $O_\gamma^G(\chi_\Omega f)$ is equal to $O_\gamma^G(f) = D(\gamma) O_\gamma^M(f^{(P)})$ for $\gamma \in \Omega$ and the orbital integral is $0$ for $\gamma \notin \Omega$. 
Therefore, the orbital integrals of the functions $\chi_\Omega \cdot (\li f^{(P)})$ and $(\chi_\Omega \li f^{(P)})$ agree.  

Recall Weyl's integration formula for the group $M$: for any $h \in \cH(M)$ we have
\begin{equation}\label{weylintegrationformula}
\Tr(h, \rho) = \sum_T \frac 1 {|W(M, T)|} \int_{T_{\textup{reg}}} \Delta_M(t)^2 \theta_\rho(t) O_t(h) \dd t, 
\end{equation}
where $\theta_\rho$ is the Harish-Chandra character of $\rho$ and where $T$ runs over the Cartan subgroups of $M$ modulo $M$-conjugation, and $W(M, T)$ is the rational Weyl group of $T$ in $M$, see \cite[p.~97]{MR771670} (cf. \cite[p.~241]{MR986793}). The right hand side in Equation \eqref{weylintegrationformula} depends only on the orbital integrals of the function $h$. Thus, two functions $h, h' \in \cH(M)$ with the same orbital integrals have the same trace on all smooth $M$-representations of finite length. Therefore the $M$-trace $\Tr_M( (\chi_\Omega f)^{(P)}, \rho)$ of the function $(\chi_\Omega f)^{(P)}$ against $\rho$  is equal to $\Tr_M(\chi_\Omega (f^{(P)}), \rho)$.  By combining Equation \eqref{weylintegrationformula} with Equation \eqref{weylstepA} we obtain the proposition.
\end{proof}

\subsection{The Kottwitz functions $f_{n\alpha s}$}\label{subsection.kottwitzfunction} From this point onwards $G$ is the general linear group. Let $n$ and $\alpha$ be positive integers, and let $s$ be a non-negative integer with $s \leq n$. We  call the number $s$ the \emph{signature}, and we call the number $\alpha$ the \emph{degree}. Let $\mu_s \in X_*(T) = \Z^n$ be the cocharacter defined by
$$
(\underset{s}{\underbrace{1, 1, \ldots, 1}}, \underset{n-s}{\underbrace{0, 0, \ldots, 0}}) \in \Z^n.
$$ 
We write $A_n$ for the algebra $\C[X_1^{\pm 1}, \ldots, X_n^{\pm 1}]^{\iS_n}$. The function $f_{n \alpha s} \in \cH_0(G)$ is the spherical function with
\begin{align*}
\cS_G(f_{n\alpha s}) &= q^{\alpha s (n-s)/2} \sum_{\nu \in \iS_n \cdot  \mu_s } [\nu]^\alpha = q^{\alpha s (n-s)/2} \sum_{I \subset \{1, \ldots, n\}, \#I = s} \prod_{i \in I} X_i^\alpha \in A_n
\end{align*}
as Satake transform (cf. \cite{MR761308}). When $n,\alpha,s \in \Z_{\geq 0}$ are such that $n < s$, then we put $f_{n \alpha s} = 0$. 

\begin{definition}
Let $X = X_1^{e_1} X_2^{e_2} \cdots X_n^{e_n} \in \C[X_1^{\pm 1}, \ldots,X_n^{\pm 1}]$ be a monomial. Then the \emph{degree} of $X$ is $\sum_{i=1}^n e_i \in \Z$. We call an element of the algebra $\C[X_1^{\pm 1}, \ldots, X_n^{\pm 1}]$ homogeneous of degree $d$ if it is a linear combination of monomials of degree $d$. These notions extend to the algebras $\cH_0(G)$ and $A_n$ via the isomorphism $\cH_0(G) = A_n$ and the inclusion $A_n \subset \C[X_1^{\pm 1}, \ldots, X_n^{\pm 1}]$. 
\end{definition}

\begin{lemma}\label{generaldegreelemma}
Let $f \in \cH_0(G)$ be a homogeneous function of degree $d$. Then $f$ is supported on the set of elements $g \in G$ with $|\det g| = q^{-d}$. 
\end{lemma}
\begin{proof} (cf. \cite[p.~34 bottom]{MR1007299}). The function $f^{(P_0)}$ is supported on the set of elements $t \in T$ with $|\det t| = q^{-d}$. Let $\chi$ be the characteristic function of the subset $\{g \in G\,\,|\, \,|\det g| = q^{-d}\} \subset G$. The Satake transform $(\chi f)^{(P_0)}$ is equal to $\chi|_T \cdot (f^{(P_0)})$. The function $\chi f$ is equal to $f$ by injectivity of the Satake transform. 
\end{proof}

By taking $f = f_{n\alpha s}$ we obtain in particular:

\begin{lemma}\label{normofsupport}
The function $f_{n\alpha s}$ is supported on the set of elements $g \in G$ with $|\det g| = q^{-\alpha s}$. 
\end{lemma}
\begin{proof}
 The Satake transform $\cS_G(f_{n\alpha s})$ of the Kottwitz function $f_{n\alpha s}$ is homogeneous of degree $\alpha s$ in the algebra $A_n$.
\end{proof}
 
\begin{lemma}\label{constantterm}
Let $P = MN $ be a standard parabolic subgroup of $G$ corresponding to the composition $(n_a)$ of $n$. Let $k$ be the length of this composition. 
The constant term of $f_{n\alpha s}$ at $P$ is equal to 
\begin{equation}\label{groot}
\sum_{(s_a)} q^{\alpha \cdot C(n_a, s_a)} \cdot \lhk f_{n_1 \alpha s_1} \otimes f_{n_2 \alpha s_2} \otimes \cdots \otimes f_{n_k \alpha s_k}\rhk, 
\end{equation}
where the sum ranges over all extended composition $(s_a)$ of $s$ of length $k$. The constant $C(n_a, s_a)$ is equal to $\frac{s(n-s)}2 - \sum_{a = 1}^k \frac {s_a (n_a - s_a)}2$.
\end{lemma}
\begin{remark}
In the above sum only the extended composition $(s_a)$ of $s$ with $s_a \leq n_a$ participate: If $s_a > n_a$ for some $a$, then $f_{n_a\alpha s_a} = 0$ by our convention.
\end{remark}
\begin{proof} (cf. \cite[Prop. 4.2.1]{MR2567740}). Let $I_a \subset \{1, 2, \ldots, n\}$ be the blocks corresponding to the composition $(n_a)$. If $I$ is a subset of the index set $\{1, \ldots, n\}$, then we  write $X_I$ for the monomial $\prod_{i \in I} X_i \in \C[X_1, X_2, \ldots, X_n]$ in this proof. 
Taking constant terms is transitive and the constant term of a spherical function is spherical. Therefore it suffices to prove that both $f_{n\alpha s}$ and the function in Equation \eqref{groot} have the same Satake transform. We compute
\begin{align*}
\sum_{(s_a)} q^{\alpha \cdot C(n_a, s_a)} \prod_{a = 1}^k \cS_G(f_{n_a \alpha s_a}) &= \sum_{(s_a)} q^{\alpha \cdot C(n_a, s_a)} \prod_{a = 1}^k q^{\alpha \frac{s_a (n_a - s_a)}2} \sum_{I \subset I_a, \# I = s_a} X_I^\alpha \cr 
&= q^{\alpha \frac{s(n-s)}2 } \sum_{(s_a)}  \prod_{a = 1}^k \sum_{I \subset I_a, \# I = s_a} X_I^\alpha \cr
\phantom{\sum_{(s_a)} q^{\alpha \cdot C(n_a, s_a)} \prod_{a = 1}^k \cS_G(f_{n_a s_a})}
&= q^{\alpha \frac{s(n-s)}2 } \sum_{(s_a)} \sum_I X_I^\alpha \quad\quad (I \subset \{1, \ldots, n\},\ \forall a:\  |I \cap I_a| = s_a ) \cr
&= q^{\alpha \frac{s(n-s)}2 } \sum_{I \subset \{1, \ldots, n\}, |I| = s} X_I^\alpha.  
\end{align*}
This concludes the proof. 
\end{proof}

\subsection{Truncation of the constant terms} In this subsection we compute the truncated function $\chi_c^G (f_{n\alpha s}^{(P)})$. This result is crucial to determine which representations of $G$ contribute to the cohomology of the basic stratum of Shimura varieties associated to unitary groups. 

\begin{proposition}\label{chicfp} 
Let $P = MN$ be a standard parabolic subgroup of $G$, and let $(n_a)$ be the corresponding composition of $n$. Let $k$ be the length of the composition $(n_a)$ and let $d$ be the greatest common divisor of $n$ and $s$. The truncated constant term $\chi_c^G ( f^{(P)}_{n\alpha s})$ is non-zero only if there exists a composition $(d_a)$ of $d$ such that for all indices $a$ the number $n_a$ is obtained from $d_a$ by multiplying with $\frac nd$. If such a composition $(d_a)$ exists, then the function $\chi_c^G (f_{n\alpha s}^{(P)})$ is equal to
\begin{equation}
\chi_c^G (f_{n\alpha s}^{(P)}) =  q^{\alpha \cdot C(n_a, s_a)} \cdot \lhk \chi_c^{G_{n_1}} f_{n_1 \alpha s_1} \otimes \chi_c^{G_{n_2}} f_{n_2  \alpha s_2} \otimes \cdots \otimes \chi_c^{G_{n_k}} f_{n_k \alpha  s_k}\rhk \in \cH_0(M), 
\end{equation}
where $s_a = \frac sd \cdot d_a$ for all $a \in \{1, 2, \ldots, k\}$, and the constant $C(n_a, s_a)$ equals $\frac{s(n-s)}2 - \sum_{a = 1}^k \frac {s_a (n_a - s_a)}2$. 
\end{proposition}
\begin{proof} By Lemma \ref{constantterm} the truncated constant term $\chi_c^G (f_{n\alpha s}^{(P)})$ is a sum of terms of the form $\chi_c^G (f_{n_1\alpha s_1} \otimes \cdots \otimes f_{n_k \alpha s_k})$ where $(s_a)$ ranges over extended compositions of $s$. To prove the Proposition we describe precisely the extended compositions with non-zero contribution. Thus assume that one of those terms is non-zero; say the one which corresponds to the extended composition $(s_a)$ of $s$. Let $m$ be a semisimple point in  $M$ at which this term does not vanish. Let $m_a \in G_{n_a}$ be the $a$-th block of $m$, and let $m_{a, 1}, \ldots, m_{a, n_1} \in \li F$ be the set of eigenvalues of $m_a$. The element $m$ is compact not only in the group $M$, but also in the group $G$, and therefore the absolute value $|m_{a, i}|$ is equal to the absolute value $|m_{b, j}|$ for all indices $a,i,b$ and $j$. In particular the value $| \det(m_a) |^{1/n_a}$ is equal to $|\det(m_b)|^{1/n_b}$. By Lemma \ref{normofsupport} the absolute value of the determinant $\det(m_a)$ is equal to $q^{-\alpha  s_a}$. Therefore the fraction $\frac{s_a}{n_a}$ is equal to the fraction $\frac {s_b}{n_b}$ for all indices $a$ and $b$. We claim that the fraction $\frac{s_a}{n_a}$ equals $\frac sn$. To see this, we have $n_b \frac{s_a}{n_a} = s_b$ for all indices $a,b$, and thus
\begin{equation}\label{rekentrucje}
n \frac {s_a}{n_a} = (n_1 + n_2 + \ldots + n_k) \frac{s_a}{n_a} = s_1 + s_2 + \ldots + s_k = s,
\end{equation}
which proves the claim. We have $\frac ds \cdot s_a = \frac ds \cdot n_a \cdot \frac sn = \frac dn \cdot n_a$. Because  $n_a \cdot \frac sn = s_a$ is integral, the number $n_a \frac dn = s_a \frac ds$ is integral as well. This implies that the composition $(n_a)$ (resp. $(s_a)$) is obtained from the composition $(d_a) := (n_a \frac dn)$ by multiplying with $\frac nd$ (resp. $\frac sd$).
\end{proof}

\subsection{The functions $\chi_N f^{(P)}_{n \alpha s}$ and $\widehat\chi_N f^{(P)}_{n\alpha s}$}\label{obtusefunctionexplicitgln}
Let $P = MN$ be a standard parabolic subgroup of $G$.
The functions $\chi_N f_{n\alpha s}^{(P)}$ and $\widehat \chi_N f^{(P)}_{n\alpha s}$ occur in the formulas for the compact traces on smooth representations of $G$ of finite length (see Proposition \ref{clozelsformula} and Proposition \ref{clozelwaldspurger}). For later computations it will be useful to have them determined explicitly. 

\begin{proposition}\label{glnactuelemma}\label{glnobtuselemma}
Let $P = MN$ be a standard parabolic subgroup of $G$, and let $(n_a)$ be the corresponding composition of $n$. Write $k$ for the length of the composition $(n_a)$. The following statements are true:
\begin{enumerate}
\item[(\textit{i})] The function $\chi_N f_{n \alpha s}^{(P)} \in \cH_0(M)$ is equal to 
$$
\sum_{(s_a)} q^{\alpha \cdot C(n_a, s_a)} \cdot \lhk f_{n_1 s_1} \otimes f_{n_2 s_2} \otimes \cdots \otimes f_{n_k s_k}\rhk, 
$$
where the sum ranges over all extended compositions $(s_a)$ of $s$ of length $k$ satisfying 
$$
\frac{s_1}{n_1} > \frac {s_2}{n_2} > \ldots > \frac{s_k}{n_k}. 
$$ 
\item[(\textit{ii})] The function $\widehat \chi_N f_{n\alpha s}^{(P)} \in \cH_0(M)$ is equal to 
$$
\sum_{(s_a)} q^{\alpha C(n_a, s_a)} \cdot \lhk f_{n_1 \alpha s_1} \otimes f_{n_2\alpha s_2} \otimes \cdots \otimes f_{n_k \alpha s_k}\rhk, 
$$
where the sum ranges over all extended compositions $(s_a)$ of $s$ of length $k$ satisfying 
$$
 (s_1 + s_2 + \ldots + s_a) > \frac sn (n_1 + n_2 + \ldots + n_a),
$$
for all indices $a$ strictly smaller than $k$.
\end{enumerate}
\end{proposition}
\begin{proof}
Let $H_i$ for $i = \{1, 2, \ldots, n\}$ denote the $i$-th vector of the canonical basis of the vector space $\ia_0 =\R^n$. 
The subset $\Delta_P$ of $\Delta$ is the subset consisting of the roots $\alpha_{n_1 + n_2 + \ldots + n_a}$ for $a \in \{1, 2, \ldots, k-1\}$. For any root $\alpha = \alpha_i |_{\ia_P}$ in $\Delta_P$ we have: 
\begin{equation}\label{fundamentalrts}
\varpi_\alpha^G = (H_1 + \cdots + H_i - \frac in (H_1 + H_2 + \cdots + H_n))|_{\ia_P}.
\end{equation}
Let $m$ be an element of the standard Levi subgroup $M$. By Lemma \ref{conditieherschrijven} the element $m$ lies in the obtuse Weyl chamber if and only if the absolute value $|\varpi_\alpha^G(m)|$ is smaller than $1$ for all roots $\alpha$ in $\Delta_P$. By Equation \eqref{fundamentalrts} the evaluation at $m = (m_a)$ of the characteristic function $\widehat \chi_N(m)$ is equal to $1$ if and only if 
\begin{equation}\label{lemconA}
|\det(m_1)| \cdot |\det(m_2)| \cdots |\det(m_a)| < |\det(m)|^{\frac {n_1 + n_2 + \ldots + n_a}n}
\end{equation}
for all indices $a \in \{1, \ldots, k-1\}$. 

We determine the function $\widehat \chi_N f^{(P)}_{n\alpha s}$. Let $m = (m_a)$ be an element of $M$. Assume $m$ lies in the obtuse Weyl chamber (cf. Equation \eqref{lemconA}). Let $(s_a)$ be an extended composition of $s$. Besides the condition $\widehat \chi_N(m) \neq 0$ we assume that
$\lhk f_{n_1 \alpha s_1} \otimes f_{n_2 \alpha s_2} \otimes \cdots \otimes f_{n_k \alpha s_k}\rhk(m) \neq 0$. By Lemma \eqref{normofsupport}
the absolute value $|\det(m_a)|$ is equal to $q^{-s_a \alpha}$ for all indices $a$. By  Equation \eqref{lemconA} we thus have the equivalent condition
\begin{equation}\label{lemconC}
(s_1 + s_2 + \ldots + s_a) > \frac sn (n_1 + n_2 + \cdots + n_a )
\end{equation}
for all indices $a \in \{1, \ldots, k-1\}$. We have proved that if the product of the obtuse function $\widehat \chi_N$ with the function $\lhk f_{n_1 \alpha s_1} \otimes f_{n_2 \alpha s_2} \otimes \cdots \otimes f_{n_k \alpha s_k}\rhk$ is non-zero, then the extended composition $(s_a)$ satisfies Equation \eqref{lemconC} for all indices $a < k$. 
Conversely, if the extended composition $(s_a)$ satisfies the conditions in Equation \eqref{lemconC}, then any element $m$ of $M$ with $|\det(m_a)| = q^{-s_a \alpha}$  satisfies $\widehat \chi_N(m) = 1$. This completes the proof of the proposition for the function $\widehat \chi_N f^{(P)}_{n\alpha s}$.

The proof for the function $\chi_N f^{(P)}_{n\alpha s}$ is the same: Instead of using  Equation \eqref{lemconA}, one uses that $\chi_N(m)$ equals $1$ if and only if $|\alpha(m)| < 1$ for all roots $\alpha \in \Delta_P$. Therefore the element $m$ lies in the acute Weyl chamber if and only if 
\begin{equation}\label{acuteexplicitGLn}
|\det(m_1)|^{1/n_1} < |\det(m_2)|^{1/n_2} < \cdots < |\det(m_k)|^{1/n_k}. 
\end{equation}
This completes the proof.
\end{proof}

\subsection{Computation of some compact traces}\label{trivialrepresentation} In this subsection we compute compact traces against the trivial representation and the Steinberg representation. 

\begin{definition}
If $\pi$ is an unramified representation of some Levi subgroup $M$ of $G$ then we write $\varphi_{M, \pi} \in \widehat M$ for the \emph{Hecke matrix} of this representation. We recall the definition of the Hecke matrix. For an unramified representation $\pi$ of $G$ there exists a smooth unramified character $\chi$ of the torus $T$ and a surjection $\Ind_{P_0}^G(\chi) \surjects \pi$. Fix such a character $\chi$ together with such a surjection. Let $\widehat T$ be the complex torus dual to $T$. We compose any rational cocharacter $F^\times \to T(F)$ with $\chi$, and then we evaluate this composition at the prime element $\varpi_F$. This yields an element of $\Hom(X_*(T), \C^\times)$. The set $\Hom(X_*(T), \C^\times)$ is equal to the set $X_*(\widehat T) \otimes \C^\times = \widehat T(\C)$. Thus we have an element of $\widehat T(\C)$ well-defined up to the action of the rational Weyl-group of $T$ in $M$. This element in $\widehat T(\C)$ is the \emph{Hecke matrix} $\varphi_{M, \pi} \in \widehat M$. 
\end{definition}

\begin{proposition}\label{ctsteinberg}
Let $f \in \cH_0(G)$ be a spherical function on $G$. 
Let $\St_G$ be the Steinberg representation of $G$. The compact trace $\Tr(\chi_c^G f, \St_G)$ is equal to $\eps_{P_0} \cS_T(\widehat \chi_{N_0} f^{(P_0)})(\varphi_{T, \delta_{P_0}^{1/2}})$. 
\end{proposition}
\begin{proof}
By Proposition \ref{clozelwaldspurger} we have
$$
\Tr(\chi_c^G f, \pi) = \sum_{P = MN} \eps_P \Tr(\widehat \chi_N f^{(P)}, (\St_G)_N(\delta_P^{-1/2})). 
$$
The normalized Jacquet module $(\St_G)_N(\delta_P^{-1/2})$ at a standard parabolic subgroup $P =MN$ is equal  to an unramified twist of the Steinberg representation of $M$ (cf. \cite[thm 1.7(2)]{MR1285969}). Assume that the parabolic subgroup $P = MN \subset G$ is not the Borel subgroup. Then the representation $\St_M$ of $M$ is ramified while the function $\widehat \chi_N f^{(P)}$ is spherical. The contribution of $P$ thus vanishes and consequently only the term corresponding to $P_0$ remains in the above formula. The Jacquet module $(\St_G){N_0}$ is equal to $\one(\delta_{P_0})$. This completes the proof. 
\end{proof}

\begin{lemma}\label{constantvanish}
Let $P = MN$ be a standard parabolic subgroup of $G$ which is proper. 
Let $f \in \cH_0(G)$ be a homogeneous spherical function of degree coprime to $n$. Then $\chi_c^G f^{(P)} = 0$. 
\end{lemma}
\begin{proof} Write $s$ for the degree of $f$. Let $(n_a)$ be the composition of $n$ corresponding to $P$. We may write $\chi_c^G = \chi_c^M \chi_M^G$ as functions on $M$, where $\chi_M^G \in C^\infty(M)$ is the characteristic function of the set of elements $m = (m_a) \in M = \prod_{a=1}^k G_{n_a}$ such that
\begin{equation}\label{mzaken}
|\det m_1|^{1/n_1} = |\det m_2 |^{1/n_2} = \cdots = |\det m_k|^{1/n_k}.
\end{equation}
We claim that $\chi_M^G f^{(P)} = 0$. 
Let $m = (m_a) \in M$ be an element such that $f^{(P)}(m) \neq 0$ and $\chi_M^G(m) \neq 0$. Thus Equation \eqref{mzaken} is true for $(m_a)$. Let $s_a$ be the integer such that $|\det m_a| = q^{-s_a}$. From Equation \eqref{mzaken} we obtain that $\tfrac {s_a}{n_a} = \tfrac {s_b}{n_b}$ for all indices $a$ and $b$. We have $s_1 + s_2 + \ldots + s_k = s$. Use the argument at Equation \eqref{rekentrucje} to obtain $\tfrac {s_a}{n_a} = \tfrac sn$ for all indices $a$. We find in particular that $n_a \tfrac sn$ is an integer. Because $n$ and $s$ are coprime this implies that $n_a = n$, \ie that $P = G$. This completes the proof. 
\end{proof}

\begin{proposition}\label{cttrivial}
Let $f \in \cH_0(G)$ be a homogeneous function of degree $s$. Assume $s$ is prime to $n$. The compact trace $\Tr(\chi_c^G f, \one)$ is equal to $\eps_{P_0} \Tr(\chi_c^G f, \St_G)$.
\end{proposition}
\begin{proof} For the trivial representation $\one$ of $G$ we have the character identity $\one = \sum_{P = MN} \eps_P\eps_{P_0} \Ind_P^G(\St_M(\delta_P^{-1/2}))$ holding in the Grothendieck group of $G$.  By the Proposition \ref{constantterminduction} we have $$\Tr(\chi_c^G f, \Ind_P^G(\St_M(\delta_P^{-1/2}))) = \Tr(\chi_c^G f^{(P)}, \St_M(\delta_P^{-1/2})).$$ By Lemma \ref{constantvanish} we have $\chi_c^G f^{(P)} = 0$ if $P$ is proper. The statement follows.
\end{proof}

\begin{example}
In this example we deduce the formula that we gave in the introduction (see Equation \eqref{introsimpleformula}). We claim that the polynomial $\cS_T(\chi_{N_0} f^{(P_0)}_{n\alpha s})$ in the ring $\C[X_1^{\pm 1}, X_2^{\pm 1}, \ldots X_n^{\pm 1}]$ is equal to the polynomial
$q^{\alpha \frac {s(n-s)}2 } \sum X_{i_1}^\alpha X_{i_2}^\alpha \cdots X_{i_s}^\alpha $ where the indices $i_1, i_2, \ldots, i_s$ in the sum range over the set $\{1, 2, \ldots, n\}$ and satisfy the conditions
(1) $i_1 < i_2 < i_3 < \ldots < i_s$; (2) $i_1 = 1$; (3) If $s > 1$ then for each subindex $j \in \{2, \ldots, s\}$ we have $i_j < 1+ \frac ns (j-1)$. The verification is elementary from Equation \eqref{satakeoff} but let us give details anyway. Let $(s_i)$ be an extended composition of $s$ of length $n$ with $s_i \in \{0, 1\}$ for all $i$ and assume that the monomial  $M(s_i) := X_1^{\alpha s_1} X_2^{\alpha s_2} \cdots X_n^{\alpha s_n}$ occurs in $\cS_T(\chi_{N_0} f^{(P_0)}_{n\alpha s})$ with a non-zero coefficient. We have 
\begin{equation}\label{obtuseS}
s_1 + s_2 + \ldots + s_i > \frac sn i
\end{equation}
for all $i < n$. Define for each subindex $j \leq s$ the index $i_j$ to be equal to $\inf\{i : s_1 + s_2 + \ldots + s_i = j\}$. With this choice for $i_j$ we have $M(s_i) = X_{i_1}^\alpha X_{i_2}^\alpha \cdots X_{i_s}^\alpha$. Equation \eqref{obtuseS} forces  $i_1 = 1$ and for all $j \leq s-1$ that $i_{j+1} - 1$ is equal to the supremum $\sup\{i : s_1 + s_2 + \ldots + s_i = j \}$. Consequently $j > \frac sn (i_{j+1} - 1)$ for all $j \leq  s-1$. By replacing $j$ by $j-1$ in this last formula we obtain for all $j$ with $2 \leq j \leq s$ the inequality 
$$
i_j < 1 + (j-1) \frac ns. 
$$
In the inverse direction, starting from this inequality for all $j$ together with the condition ``$i_1 = 1$'' we may go back to the inequalities in Equation \eqref{obtuseS}. This proves the claim. 
\end{example}

\begin{example} 
We have
\begin{align*}
\Tr(\chi_c^G f_{n\alpha 1}, \one) &= 1 \cr 
\Tr(\chi_c^G f_{n\alpha 2}, \one) &= 1 + q^\alpha + q^{2\alpha} + \ldots + q^{\alpha(\lfloor \frac n2 \rfloor -1)}. 
\end{align*}
\end{example}

\newcommand{\Speh}{\textup{Speh}}
\section{Discrete automorphic representations and compact traces}\label{sectionrigid}

We introduce two classes of semi-stable representations, the Speh representationds and the rigid representations which are certain products of Speh representations. Then we deduce from the Moeglin-Waldspurger classification the possible components at $p$ of discrete automorphic representations in the semi-stable case. 

Let $x, y$ be integers such that $n = xy$. We define the representation $\Speh(x,y)$ of $G$ to be the unique irreducible quotient of the representation $|\det|^{\frac {y-1}2} \St_{G_x} \times |\det |^{\frac {y-3}2} \St_{G_x} \times \cdots \times |\det |^{-\frac{y-1}2} \St_{G_x}$ where the product means unitary parabolic induction from the standard parabolic subgroup of $G_n$ with $y$ blocks and each block of size $x$. A \emph{semi-stable Speh representation} of $G$ is, by definition, a representation isomorphic to $\Speh(x, y)$ for some $x,y$ with $n = xy$. We emphasize that we did not introduce all Speh representations, we have introduced only the ones which are semi-stable. 

A smooth representation $\pi_p$ of $G$ is called \emph{semi-stable rigid} representation if it is isomorphic to a representation of the following form. 
Consider the following list of data
\begin{enumerate}
\item[$\bullet$] $k \in \lbr 1, 2, \ldots, n\rbr$;
\item[$\bullet$] for each $a \in \lbr 1, 2, \ldots, k \rbr$ an unitary unramified character $\eps_a \colon G \to \C^\times$;
\item[$\bullet$] for each $a \in \lbr 1, 2, \ldots, k \rbr$ a real number $e_a$ in the open (real) interval $(-\frac 12, \frac 12)$; 
\item[$\bullet$] positive integers $y, x_1, x_2, \ldots, x_k$ such that $\frac ny = \sum_{a=1}^k x_a$, 
\end{enumerate}
then we form the representation 
$$
\Ind_{P}^{G} \bigotimes_{a=1}^k \Speh(x_a, y)(\eps_a |\cdot|^{e_a}), 
$$ where $P = MN \subset G$ is the parabolic subgroup corresponding to the composition $(yx_a)$ of $n$ and where the tensor product is taken along the blocks of $M = \prod_{a=1}^k G_{y x_a}$. We remark that these representations are irreducible. 

\begin{theorem}[Moeglin-Waldspurger]\label{temperedargument}
Let $F$ be a number field and let $v$ be a finite place of $F$.  Let $\pi_v$ be the local factor at $v$ of a discrete (unitary) automorphic representation $\pi$ of $\Gl_{n}(\A_F)$. Assume that $\pi_v$ is semi-stable. Then $\pi_v$ is a semi-stable rigid representation. 
\end{theorem}

\begin{remark}
Let $\pi_v$ be a semi-stable component of a discrete automorphic representation, as considered in the Theorem \ref{temperedargument}. Then using the definition of rigid representation we associate, among other data, to $\pi_v$ the real numbers $e_a$ in the open interval $(-\frac 12, \frac 12)$ (see above).   
The Ramanujan conjecture predicts that the numbers $e_a$ are $0$. This conjecture is proved in the restricted setting of Section \S \ref{section.mainproof} where we work with automorphic representations occurring in the cohomology of certain Shimura varieties. Therefore the numbers $e_a$, which a priori could be there, will not play a role for us. 
\end{remark}

\begin{proof}[Proof of Theorem \ref{temperedargument}]
By the classification of the discrete spectrum of $\Gl_n(\A_F)$ in \cite{MR1026752} there exist 
\begin{enumerate}
\item[$\bullet$] a decomposition $n = xy$, $x,y\in \Z_{\geq 1}$;
\item[$\bullet$] a cuspidal automorphic representation $\omega$ of $\Gl_x(\A_F)$;
\item[$\bullet$] a character $\eps \colon \Gl_n(\A_F) \to \C^\times$,
\end{enumerate}
such that after twisting by $\eps$, the representation $\pi$ is the irreducible quotient $J$ of the induced representation $I$ which is equal to $\ind_{P_x(\A_F)}^{\Gl_n(\A_F)} \lhk \omega |\cdot |^{\frac {y-1} 2}, \ldots, \omega |\cdot |^{\frac{1-y}2} \rhk$. In this formula the induction is unitary and $P_x = M_x N_x \subset \Gl_n$ is the standard parabolic subgroup of $\GL_n$ with $y$-blocks, each one of size $x \times x$. By applying the local component functor~\cite[prop 2.4.1]{MR1265561} to the surjection $I \surjects J$ we obtain a surjection $I_v \surjects J_v$. The component at $v$ of $I_v$ is simply $\ind_{P_x(F_v)}^{\Gl_n(F_v)} \lhk \omega_v |\cdot |^{\frac {y-1} 2}, \ldots, \omega_v |\cdot |^{\frac{1-y}2} \rhk$.  The representation $\omega_v$ is a factor of a cuspidal automorphic representation of $\Gl_x(\A_F)$ and therefore \emph{generic}\footnote{This follows from the results in \cite{MR348047}, combined with the method in \cite{MR401654}, see the discussion on the end of page 172 and beginning of page 173 in the introduction to \cite{MR348047}.}.

From this point onwards we work locally at $v$ only, so we drop the $\Gl_n(F_v)$-notation and write simply $G_n$. By the Zelevinsky classification of $p$-adic representations \cite{MR584084} any generic representation is of the form $ \sigma_1 |\det |^{e_1} \times  \sigma_2|\det |^{e_2} \times \cdots \times  \sigma_k|\det |^{e_k}$ where the $\sigma_a$ are square integrable representations and the $e_a \in \R$ lie in the open interval $(-\frac 12, \frac 12)$. The $\sigma_a$ are equal to the unique irreducible subquotient of a representation of the form $\rho \times \rho |\det |^2 \times \cdots \times \rho |\det |^{k-1}$ where $\rho$ is cuspidal and where the central character of $\rho |\det |^{\frac {k-1} 2}$ is unitary. We assumed that $\pi_v$ is semi-stable. Therefore $\rho$ is semi-stable and cuspidal, and therefore a one-dimensional unramified character. This implies that $\sigma_a$ is equal to $\St_{G_{n_a}}(\eps_a)$ for some $n_a \in \Z_{\geq 0}$ and some unramified unitary character $\eps_a$ of $G_{n_a}$. Thus $\sigma$ is equal to $\St_{G_{n_1}}(\eps_1) |\det |^{e_1} \times \cdots \times \St_{G_{n_k}}(\eps_k) |\det |^{e_k}$. For the representation $I_v$ we obtain
\begin{align*}
I_v &= \omega_v |\det |^{\frac {y-1}2} \times \cdots \times \omega_v |\det |^{\frac {1-y}2} \cr
&= \lhk \sigma_1 |\det |^{e_1} \times \cdots \sigma_r | \det |^{e_r} \rhk|\det |^{\frac {y-1}2} \times \cdots \times \lhk \sigma_1 |\det |^{e_1} \times \cdots \sigma_r | \det |^{e_r} \rhk |\det |^{\frac {1-y}2} \cr
&= \prod_{a=1}^k \lhk \sigma_a |\det |^{e_a + \frac{y-1}2} \times \cdots \times \sigma_i |\det |^{e_a + \frac{1-y}2} \rhk.
\end{align*}
For each $i$, the representation $\sigma_a |\det |^{e_a + \frac{y-1}2} \times \cdots \times \sigma_a |\det |^{e_a + \frac{1-y}2}$ has $\Speh(x_a, y)(\eps_a |\det |^{e_a})$ as (unique) irreducible quotient. Thus we obtain a surjection $I_v \surjects \prod_{a=1}^k \Speh(x_a, y)(\eps_a |\det |^{e_a})$.
The representations $\Speh(x_a, y)(\eps_a)$ are unitary and because $|e_a|$ is \emph{strictly} smaller than $\frac 12$ it is impossible to have a couple of indices $(a,b)$ such that the representation $\Speh(x_a, y)(\eps_a |\det |^{e_a})$ is a twist of $\Speh(x_b, y) (\eps_b |\det |^{e_b})$ with $|\det |$. By the Zelevinsky segment classification it follows that the product 
$\prod_{a=1}^k \Speh(x_a, y)(\eps_a |\det |^{e_a})$ is irreducible. By uniqueness of the Langlands quotient the representation $J_v$ is isomorphic to the product $\prod_{a=1}^k \Speh(x_a, y)(\eps_a |\det |^{e_a})$, as required. 
\end{proof}

\begin{proposition}\label{vanishunlesssteinberg}
Let $\pi$ be  a semi-stable rigid representation of $G = \Gl_n(F)$ where $F$ is a finite extension of $\qp$. Let $f$ be a homogeneous function in $\cH_0(G)$ of degree $s$ coprime to $n$, then the compact trace $\Tr(\chi_c^G f, \pi)$ vanishes unless $\pi$ is the trivial representation or  the Steinberg representation. 
\end{proposition}
\begin{proof}
Assume that $\Tr(\chi_c^G  f, \pi)$ is non-zero. By Proposition \ref{constantterminduction} the compact trace of $\chi_c^G f^{(P)}$ against the representation $\bigotimes_{i=1}^k \Speh(x_i, y)(\eps_a |\cdot|^{e_a})$ is non zero. The truncated constant term $\chi_c^G f^{(P)}$ vanishes if the parabolic subgroup $P \subset G$ is proper (Lemma \ref{constantvanish}). Therefore $\pi$ is a Speh-representation; say $x$ and $y$ are its parameters. The character formula of Tadic \cite[p. 342]{MR1359141} expresses $\pi$ as an alternating sum of induced representations:
$$
|\det |^{\frac {x+y}2} u(\St_x, y) = \sum_{w \in \iS_y'} \eps(w) \prod_{i=1}^y \delta[i, x+w(i) - 1] \in \cR
$$
(for notations see [\textit{loc. cit}]). The compact trace on all these induced representations vanish unless they are induced from the parabolic subgroup $P= G$. This is true only if the representation $\delta[i, x + w(i) - 1] $ is the unit element in $\cR$ for all indices expect one, \ie if
$(x + w(i) - 1) - i + 1 = 0$. After simplifying we find that  $w(i) = i - x$ for all indices $i$ except one. Make the assumption that $y > 1$. Then clearly, if $x > 1$, the number $i - x$ is non-positive for the indices $i = 1$ and $i = 2$. It then follows that $w(i)$ is non-positive for $i = 1$ or $i = 2$. However, that is impossible because $w$ is a permutation of the index set $\{1, 2, \ldots, y\}$. The conclusion is that either $y =1$ or $x = 1$. But then $\pi$ is the Steinberg or the trivial representation. 
\end{proof}

\section{The basic stratum of some Shimura varieties associated to division algebras}\label{establishmainformula}\label{section.mainproof}

In this section we establish the main result of this article. 

\subsection{Notations and assumptions} As explained in the introduction, we place ourselves in a restricted version of the setting of Kottwitz in the article  \cite{MR1163241}. We start by copying some of the notations from that article. Let $D$ be a division algebra over $\Q$ equipped with an anti-involution $*$. Let $\li \Q$ be the algebraic closure of $\Q$ inside $\C$. Write $F$ for the center of $D$ and we embed $F$ into $\li \Q$. We assume that $F$ is a CM field and we assume that $*$ induces the complex conjugation on $F$.  We write $F^+$ for the totally real subfield of $F$ and we assume that $F$ decomposes into a compositum $\cK F^+$ where $\cK/\Q$ is quadratic imaginary. Let $n$ be the positive integer such that $n^2$ is the dimension of $D$ over $F$. Let $G$ be the $\Q$-group such that for each commutative $\Q$-algebra $R$ the set $G(R)$ is equal to the set of elements $x \in D \otimes_\Q R$ with $x x^* \in R^\times$. The mapping $c \colon G \to \G_{m, \Q}$ defined by $x \mapsto xx^*$ is called the \emph{factor of similitude}. Let $h_0$ be an algebra morphism $h_0 \colon \C \to D_\R$ such that $h_0(z)^* = h_0(\li z)$ for all $z \in \C$. We assume that the involution $x \mapsto h_0(i)^{-1} x^* h_0(i)$ is positive. We restrict $h_0$ to $\C^\times$ to obtain a morphism $h$ from Deligne's torus $\Res_{\C/\R} \G_{m, \C}$ to $G_\R$; we let $X$ be the $G(\R)$ conjugacy class of $h$. The triple $(G, X, h)$ is a Shimura datum and the triple $(G, X, h^{-1})$ is also a Shimura datum. We work with the second datum, which is different from the first by a sign\footnote{The reason for this sign is that the formula conjectured by Kottwitz in the article \cite{MR1163241} turned out to be slightly mistaken. When Kottwitz proved his conjecture in \cite{MR1124982} he found that a different sign should be used. However he did not change the sign in the conclusion of his theorem, rather he introduced it at the beginning by replacing $h$ by $h^{-1}$.  We follow the conventions of Kottwitz because we refer to both articles constantly.}. Let $\mu \in X_*(G)$ be the restriction of $h\otimes \C \colon \C^\times \times \C^\times \to G(\C)$ to the factor $\C^\times$ of $\C^\times \times \C^\times$ indexed by the identity isomorphism $\C \isomto \C$. We write $E \subset \li \Q$ for the reflex field of this Shimura datum (see below for a description of $E$). We obtain varieties $\Sh_K$ defined over the field $E$ and these varieties represent corresponding moduli problems of Abelian varieties of PEL-type as defined in \cite{MR1124982}. 

Let $p$ be a prime number where the group $G_{\qp}$ is unramified over $\qp$, and the conditions of \cite[\S 5]{MR1124982} are satisfied so that the moduli problem and the variety $\Sh_K$ extend to be defined over the ring $\cO_E \otimes \Zp$ [\textit{loc. cit.}]. We assume that the prime $p$ \emph{splits} in the field $\cK$. Let $K \subset G(\Af)$ be a compact open subgroup, of the form $K = K_p K^p$, with $K_p \subset G(\qp)$ hyperspecial (coming from the choice of a lattice and extra data, see [\textit{loc. cit.}, \S 5]). Furthermore, we assume that $K^p \subset G(\Af^p)$ is small enough such that $\Sh_K / \cO_E\otimes \Zp$ is smooth [\textit{loc. cit}, \S 5]. Fix an embedding $\nu_p \colon E \to \lqp$. The embedding $\nu_p$ induces an $E$-prime $\p$ lying above $p$. We write $\fq$ for the residue field of $E$ at the prime $\p$. 

Let $\xi$ be an irreducible algebraic representation over $\li \Q$ of $G_{\li \Q}$ and let $\cL$ be the local system corresponding to $\xi \otimes \C$ on the variety $\Sh_{K, \cO_{E_\p}}$. Let $\ig$ be the Lie algebra of $G(\R)$ and let $K_\infty$ be the stabilizer subgroup in $G(\R)$ of the morphism $h$. Let $f_\infty$ be a Clozel-Delorme  pseudo-coefficient \cite{MR794744} so that the following property holds. Let $\pi_\infty$ be an $(\ig, K_\infty)$-module occurring as the component at infinity of an automorphic representation $\pi$ of $G$. Then the trace of $f_\infty$ against $\pi_\infty$ is equal to the Euler-Poincar\'e characteristic  $\sum_{i=0}^\infty (-1)^i \dim \uH^i(\ig, K_\infty; \pi_\infty \otimes \xi)$ (cf. \cite[p.~657, Lemma 3.2]{MR1163241}). Let $\ell$ be an auxiliary prime number (different from $p$) and $\lql$ an algebraic closure of $\Q_\ell$ together with an embedding $\li \Q \subset \lql$. We write $\cL$ for the $\ell$-adic local system on $\Sh_{K, \cO_{E_\p}}$ associated to the representation $\xi \otimes \lql$ of $G_{\lql}$. 
 
Because $p$ splits in the extension $\cK/\Q$, the group $G_{\qp}$ splits into a direct product of general linear groups:
\begin{equation}\label{groupisom}
G_{\qp} \cong \G_{\textup{m}, \qp} \times \prod_{\wp | p} \Res_{F^+_\wp/\qp} \Gl_{n, F^+_\wp},
\end{equation}
where the product ranges over the set of $F^+$-places above $P$. 
Observe that we wrote `$\cong$' and not `$=$'. The choice of an isomorphism amounts to the choice of, for each $F^+$-place $\wp$ of an $F$-place $\wp'$ above $\wp$. Recall that we have embedded $\cK$ into $\C$ and that $F = \cK \otimes F^+$. Therefore, we have in fact for each $\wp$ such an $\wp'$. We fix for the rest of this article in Equation \ref{groupisom} the isomorphism corresponding to this choice of $F$-primes above the $F^+$-primes above $p$. We write $T_{\qp} \subset G_{\qp}$ for the diagonal torus. Observe that the group on the right hand side of the above equation has an obvious model over $\zp$; we will write $G_{\zp}$ for this model, and we assume $K_p = G_{\zp}(\zp)$.

The field $E$ is included in the field $F$. We copy Kottwitz's description of the reflex field $E$ (cf. \cite[p.~655]{MR1163241}). Consider the subgroup consisting of the elements $g \in G$ whose factor of similitude is equal to $1$. This subgroup is obtained by Weil restriction of scalars from an unitary group $U$ defined over the field $F^+$. Let $v \colon F^+ \to \R$ be an embedding and let $v_1, v_2$ be the two embeddings of $F$ into $\C$ that extend $v$. We associate a number $n_{v_1}$ to $v_1$ and a number $n_{v_2}$ to $v_2$ such that the group $U(\R, v)$ is isomorphic to the standard real unitary group $U(n_{v_1}, n_{v_2})$. The group $\Aut(\C/\Q)$ acts on the set of $\Z$-valued functions on $\Hom(F, \C)$ by translations. The reflex field $E$ is the fixed field of the stabilizer subgroup in $\Aut(\C/\Q)$ of the function $v \mapsto n_v$. 

We write $V(F^+) := \Hom(F^+, \li \Q)$. We identify $V(F^+)$ with $\Hom(F^+, \lqp)$ via the embedding $\nu_p$, and also with $\Hom(F^+, \R)$ via the inclusion $F^+ \subset \R$. In particular $V(F^+)$ is a $\Gal(\lqp/\qp)$-set and a $\Gal(\C/\R)$-set. For every $F^+$-prime $\wp$ above $p$ we write $V(\wp)$ for the Galois orbit in $V(F^+)$ corresponding to $\wp$. 

We have embedded the field $\cK$ into $\C$, and thus each $\Gal(\C/\R)$-orbit in $V(F^+)$ contains a distinguished point, \ie for each each embedding $v \colon F^+ \to \C$ we have a distinguished extension $v_1 \colon F \to \C$. We write $s_v$ for the number $n_{v_1}$. We define $s_{\wp} := \sum_{v \in V(\wp)} s_v$. We define $\Unr_p$ to be the set of $F^+$-places $\wp$ above $p$ such that $s_{\wp} = 0$, and $\Ram_p$ to be the set of $F^+$-places above $p$ such that $s_{\wp} > 0$. We work under one additional technical assumption: We assume that for every $\wp \in \Ram_p$ the number $s_{\wp}$ is coprime with $n$. 

\subsection{Isocrystals and the basic stratum}\label{subsectionisocrystals} 
Write $\cA_K$ for the universal Abelian variety over $\Sh_K$ and $\lambda, i, \li \eta$ for its additional PEL type structures \cite[\S 6]{MR1124982}. Let $L$ be the completion of the maximal unramified extension of $\qp$ contained in $\lqp$. Then $E_{\p, \alpha}$ is a subfield of $L$. We write $\sigma$ for the automorphism of $L$ which acts by $x \mapsto x^p$ on the residue field of $L$. We write $V$ for the $D^\textup{opp}$-module with space $D$ where an element $d \in D^{\opp}$ acts on the left through multiplication on the right on the space $D$.

Let $x \in \Sh_K(\fqa)$ be a point. The rational Dieundonn\'e module $\D(\cA_{K,x})_\Q$ is an $(E_{\p, \alpha}/\qp)$-isocrystal. The couple $(\lambda, i)$ induces via the functor $\D(\square)_\Q$ additional structures on this isocrystal. There exists an isomorphism $\varphi \colon V \otimes E_{\p, \alpha} \isomto \D(\cA_{K,x})_\Q$ of skew-Hermitian $B$-modules \cite[p.~430]{MR1124982}, and via this isomorphism we can send the crystalline Frobenius on $\D(\cA_{K,x})_\Q$ to a $\sigma$-linear operator on $V \otimes E_{\p, \alpha}$. This operator on $V \otimes E_{\p, \alpha}$ may be written in the form $\delta \cdot (\id_V \otimes \sigma)$ where $\delta \in G(E_{\p, \alpha})$ is independent of $\varphi$ up to $\sigma$-conjugacy. We also have the $L$-isocrystal $\D(\cA_{K,x} \otimes \lfq)_\Q = \D(\cA_{K,x})_L$ which induces in the same manner an element of $G(L)$, well defined up to $\sigma$-conjugacy. Let $B(G_{\qp})$ be the set of all $\sigma$-conjugacy classes in $G(L)$ from \cite{MR809866}. This set classifies the $L$ isocrystals with additional $G_{\qp}$-structure up to isomorphism. 

In the articles \cite{MR1411570} and \cite{MR1485921} there is introduced the subset $B(G_{\qp}, \mu_{\lqp}) \subset B(G_{\qp})$ of $\mu_{\lqp}$-admissible isocrystals. The point is that if an isocrystal arises from some element $x \in \Sh_K(\lfq)$ then this isocrystal is always $\mu_{\lqp}$-admissible. The set $B(G_{\qp})$ can be described explicitly as follows. We have $G_{\qp} = \G_{\textup{m}, \qp} \times \Res_{F^+_{\qp}/\qp} \Gl_{n, F^+_{\qp}}$ which induces the decomposition
$$
B(G_{\qp}) = B(\Gm) \times \prod_{\wp | p} B(\Res_{F_\wp^+/\qp} \Gl_{n, F^+_\wp}). 
$$ 
Write $\mu_\wp$ for the component at $\wp$ of the cocharacter $\mu$. Fix one $\wp | p$. There is the Shapiro bijection \cite[Eq. 6.5.3]{MR1485921} 
$$
B(\Res_{F_\wp^+/\qp} \Gl_{n, F^+_\wp}, \mu_\wp) = B(\Gl_{n, F^+_\wp}, \mu_\wp')
$$ 
where the right hand side is the set of $\sigma^{[F^+_\wp:\qp]}$-conjugacy classes in $\Gl_n(L)$  and $\mu_\wp'$ is defined by
$$
\mu_\wp' \bydef \sum_{v \in V(\wp)} (\underset{s_v}{\underbrace{1, 1, \ldots, 1}}, \underset{n-s_v}{\underbrace{0, 0, \ldots, 0}}) \in \Z^n. 
$$
There is an unique element $b \in B(G_{\qp})$ with the property that, for each $\wp$, the corresponding isocrystal $b_\wp'$ in $B(\Gl_{n, F^+_\wp}, \mu_\wp')$ has precisely one slope (\ie $b$ is \emph{basic}). This slope must then be $\tfrac {s_\wp}n$ because the end point of the Hodge polygon of $\mu_\wp'$ is $(n, s_{\wp})$. The component of $b$ at the factor of similitude is the $\sigma$-conjugacy class equal to the set of elements $x \in L^\times$ whose valuation is equal to $1$. 

\begin{lemma}
We have $(n, s_\wp) = 1$ if and only if the isocrystal $V_\wp$ is simple. 
\end{lemma}
\begin{proof}
We have $V_\wp = V_\lambda^m$ where $V_\lambda$ is the simple object of slope $\lambda = \tfrac {s_\wp}{n}$. In case $s_\wp$ and $n$ are coprime this simple object is of height $n$; otherwise its height is strictly less than $n$, and it therefore has to occur with positive multiplicity. 
\end{proof}

The isocrystal $b$ introduced above characterises the basic stratum $B \subset \Sh_{K, \Fq}$ as the reduced subscheme such that for all points $x \in B(\lfq)$ the isocrystal associated to the Abelian variety $\cA_{K,x}$ is equal to $b$. The variety $B$ is projective, but in general not smooth\footnote{The only cases where we know it is smooth is when it is a finite variety.}. 

The Hecke correspondences on $\Sh_K$ may be restricted to the subvariety $\iota \colon B \hookrightarrow \Sh_{K,\fq}$ and we have an action of the algebra $\cH(G(\Af))$ on the cohomology spaces $\uH_\et^i(B_{\lfq}, \iota^* \cL)$. Furthermore, the space $\uH_\et^i(B_{\lfq}, \iota^* \cL)$ carries an action of the Galois group $\Gal(\lfq/\fq)$ which commutes with the action of $\cH(G(\Af))$.   
 
\subsection{The function of Kottwitz} Let $\alpha$ be a positive integer. Let $E_{\p, \alpha} /E_{\p}$ be an unramified extension of degree $\alpha$. We write $\phi_\alpha$ for the characteristic function of the double coset $G(\cO_{E_{\p, \alpha}}) \mu(p^{-1})G(\cO_{E_{\p, \alpha}})$ in $G(E_{\p, \alpha})$. The function $f_\alpha$ is by definition obtained from $\phi_\alpha$ via base change from $G(E_{\p, \alpha})$ to $G(\qp)$. We call the functions $f_\alpha$ the \emph{functions of Kottwitz}; they play a fundamental role in the point-counting formula of Kottwitz for the number of points of the variety $\Sh_K$ over finite fields. In this section we give an explicit description of these functions $f_\alpha$ of Kottwitz. 

\begin{definition}
Let $\wp$ be an $F^+$-place above $p$. We write $V_\alpha(\wp)$ for the set of $\Gal(\lqp/E_{\p, \alpha})$-orbits in the set $V(\wp)$, and $V_\alpha(F^+)$ for the set of $\Gal(\lqp/E_{\p, \alpha})$ orbits in the set $V(F^+)$. If $v \in V_\alpha(F^+)$ is such an orbit, then this orbit corresponds to a certain finite unramified extension $E_{\p, \alpha}[v]$ of $E_{\p, \alpha}$. Let $\alpha_v$ be the degree over $\qp$ of the field $E_{\p, \alpha}[v]$, we then have $E_{\p, \alpha}[v] = E_{\p, \alpha_v}$. 
\end{definition}

\begin{remark}
Let $\li v$ be an element of $V_\alpha(\wp)$, then the number $s_v$ is independent of the choice of representative $v \in \li v$.
\end{remark}
\begin{remark}
Observe that if $F^+$ is Galois over $\Q$, then all the Galois orbits in $V(F^+)$ have the same length. 
\end{remark}

\begin{proposition}\label{kottwitzfunctionexplicit}
The function $f_\alpha$ is given by
$$
f_\alpha = \one_{q^{-\alpha} \Zp^\times} \otimes \bigotimes_{\wp|p} \prod_{v \in V_\alpha(\wp)} f^{\Gl_n(F_\wp^+)}_{n \alpha_v s_v} \in \cH_0(G(\qp)),
$$
where the product is the convolution product.
\end{proposition}
\begin{proof}
We have the $\Gal(\lqp/\qp)$-set $V(F^+) = \Hom(F^+, \lqp)$. This Galois set is unramified and we have an action of the Frobenius $\sigma$. The Galois set $V(F^+)$ decomposes:
$V(F^+) = \coprod_{\wp | p} V(\wp)$, where $V(\wp) := \Hom(F^+_\wp, \lqp)$. 
We have 
$$
F^+ \otimes E_{\p, \alpha} = \prod_{\wp|p} (E_{\p, \alpha_v})^{\# V_\alpha(\wp)}.
$$
Because $p$ splits in $\cK$ we have $\cK \subset \qp \subset E_{\p, \alpha}$ and therefore
\begin{equation}\label{groupGpa}
G(E_{\p, \alpha}) = E_{\p, \alpha}^\times \times \prod_{\wp | p} \prod_{\li v \in V_\alpha(\wp)} \Gl_n(E_{\p, \alpha_v }).
\end{equation}
Recall that, by the definition of the reflex field, if two elements $v, v' \in V(F^+)$ lie in the same $\sigma^{[E_\p:\qp]}$-orbit, then $s_v = s_{v'}$. 
Thus, with respect to the decomposition in Equation \eqref{groupGpa} we may write 
$$
\phi_\alpha = \one_{p^{-1} \cO_{E_{\p, \alpha}}} \otimes \bigotimes_{\wp | p} \bigotimes_{\li v \in V_\alpha(\wp)} \one_{\Gl_n(\cO_{E_{\p, \alpha_v}})} \cdot \mu_{\li v} (p^{-1}) \cdot \one_{\Gl_n(\cO_{E_{\p, \alpha_v} })} \in \cH_0(G(E_{\p, \alpha})),
$$
where $\mu_{\li v}$ is the cocharacter 
$(\mu_{s_v})^{[E_{\p, \alpha_v}: E_{\p, \alpha}]} \in X_*(\Res_{E_{\p, \alpha_v}/E_{\p, \alpha} } \G_{\textup{m}}^n) = \Z^{n \cdot [E_{\p, \alpha_v} : E_{\p, \alpha}] }$.  

The explicit description of $f_\alpha$ now follows by applying the base change morphism from the spherical Hecke algebra of the group 
$G(E_{\p, \alpha})$ to the spherical Hecke algebra of the group $G(\qp) = \Qp^\times \times \prod_{\wp | p} \Gl_n(F^+_{\wp})$. This completes the proof. 
\end{proof}
\subsection{An automorphic description of the basic stratum}\label{mainproof} Let $\iota$ be the inclusion $B \hookrightarrow \Sh_{K, \fq}$. For each positive integer $\alpha$ and each $f^+ \in \cH(G(\Af^p))$ we define the constant
$$
T_B(f^{p}, \alpha) \bydef \sum_{i=0}^\infty (-1)^i \Tr(f^{\infty p} \times \Phi^\alpha_\p, \uH_{\et}^i(B_{\lfq}, \iota^*\cL)).
$$

We write $f$ for the function $f^{\infty p} f_\alpha f_\infty$ in the Hecke algebra of $G$ and similarly for $\chi_c^{G(\qp)} f$ even though the truncation occurs only at $p$. 

We first give an automorphic expression for the trace $T_B(f^p, \alpha)$ for all sufficiently large integers $\alpha$. 

\begin{proposition}
There exists an integer $\alpha_0$ depending on the function $f^p$ such that $T_B(f^{p}, \alpha)$ equals $\Tr(\chi_c^{G(\qp)} f, \cA(G))$ for all $\alpha \geq \alpha_0$. 
\end{proposition}
\begin{proof}
The main theorem of the article \cite{MR1124982} gives an equation of the form 
\begin{equation}\label{kottwitzequation}
\sum_{x' \in \Fix_{\Phi_\p^\alpha \times f^{\infty p}}(\lfq) } \Tr(\Phi_\p^\alpha \times f^{\infty p}, \cL_x)  = |\ker^1(\Q, G)|\sum_{(\gamma_0; \gamma, \delta)} c(\gamma_0; \gamma, \delta) O_\gamma(f^{\infty p}) TO_{\delta}(\phi_\alpha) \tr \xi_\C(\gamma_0),
\end{equation}
the notations are from [\textit{loc. cit}], see especially \S 19. We restrict this formula to the basic stratum $B$ by considering on the right hand side only \emph{basic} Kottwitz triples. In this context \emph{basic} means that the stable conjugacy class $\gamma_0$ in $(\gamma_0; \gamma, \delta)$ is compact at $p$, or, equivalently that the isocrystal corresponding to $\delta$ is the basic isocrystal in $B(G_{\qp}, \mu)$. The elements $x' \in \Fix_{\Phi_\p^\alpha \times f^{\infty p}}(\lfq)$ in the sum in the left hand side of the Equation then have to be restricted to range over the set of fix points $\Fix_{\Phi_\p^\alpha \times f^{\infty p}}^B$ of the correspondence $\Phi_\p^\alpha \times f^{\infty p}$ acting on the variety $B$. Everything else remains unchanged. This follows from the arguments of Kottwitz given for the above equation (cf. [\textit{loc. cit}, \S 19]). 

From Fujiwara's trace formula \cite[thm 5.4.5]{MR1431137} we obtain
$$
T_B(f^{p}, \alpha) = \sum_{x' \in \Fix_{\Phi_\p^\alpha \times f^{\infty p}}^B(\lfq) } \Tr(\Phi_\p^\alpha \times f^{\infty p}, \iota^*\cL_x) 
$$
for $\alpha$ large enough; say that this formula is true for all $\alpha \geq \alpha_0$. Note that in Fujiwara's statement the integer $\alpha_0$ depends on the correspondence and the sheaf $\cL$.

To any $\sigma$-conjugacy class $\delta \in G(E_{\p, \alpha})$ we  associate a conjugacy class $\cN(\delta)$ in the group $G(\qp)$ defined up to conjugacy by $\delta \sigma(\delta) \cdots \sigma^{\alpha' - 1}(\delta)$ (cf. \cite{MR1007299} \cite[p.~799]{MR683003}). The element $\delta \in G(E_{\p, \alpha})$ is called \emph{$\sigma$-compact} if its norm $\cN(g)$ is a compact conjugacy class in $G(\qp)$. Let $\chi_c^{G(\qp)}$ be the characteristic function on $G(\Qp)$ of the subset of compact elements (cf. \S \ref{subsection.compacttraces}). We let $\chi_{\sigma c}^{G(E_{\p, \alpha})}$ be the characteristic function on $G(E_{\p, \alpha})$ of the set of $\sigma$-compact elements. Consequently $T_B(f^{\infty p}, \alpha)$ is equal to  
$$
|\ker^1(\Q, G)|\sum_{(\gamma_0; \gamma, \delta)} c(\gamma_0; \gamma, \delta) O_\gamma(f^{\infty p}) TO_{\delta}(\chi_{\sigma c}^{G(E_{\p, \alpha})} \phi_\alpha)\tr \xi_\C(\gamma_0)
$$
where $(\gamma_0; \gamma, \delta)$ ranges over \emph{all} Kottwitz triples. Kottwitz has pseudo-stabilized this formula:
\begin{align}\label{stabilization.A}
\tau(G)\sum_{(\gamma_0; \gamma, \delta)} \sum_{\kappa \in \iK(I_0/\Q)} \langle \alpha(\gamma_0; \gamma, \delta),& s \rangle e(\gamma, \delta) O_\gamma(f^{\infty p}) TO_{\delta}(\chi_{\sigma c}^{G(E_{\p, \alpha})} \phi_\alpha) \tr \xi_\C(\gamma_0) \cdot \cr 
& \cdot \vol(A_G(\R)^0 \backslash I(\infty)(\R))^{-1},
\end{align}
see \cite[Eq. (7.5)]{MR1044820}. By the base change fundamental Lemma (see \cite{MR1068388} and \cite{MR868140}) the functions $\phi_\alpha$ and $f_\alpha$ have matching stable orbital integrals (the functions are \emph{associated}). By construction of the function $\chi_{\sigma c}^{G(E_{\p, \alpha})}$ this is then also the case for the truncated functions $\chi_{\sigma c}^{G(E_{\p, \alpha})} \phi_\alpha$ and $\chi_c^{G(\qp)} f_\alpha$. The group $G$ arises from a division algebra and therefore the group $\iK(G_{\gamma_0}/\Q)$ is trivial for any (semisimple) element $\gamma \in G(\Q)$ \cite[Lemma 2]{MR1163241}. Let $\gamma_\infty$ be a semisimple element of $G(\R)$. Then the stable orbital integral $\SO_{\gamma_\infty}(f_\infty)$ vanishes unless $\gamma_\infty$ is elliptic, in which case it is equal to $\vol(A_G(\R)^0\backslash I(\R))^{-1} e(I)$, where $I$ denotes the inner form of the centralizer of $\gamma_\infty$ in $G$ that is anisotropic modulo the split center $A_G$ of $G$ \cite[Lemma 3.1]{MR1163241}. Consequently Equation \eqref{stabilization.A} is equal to the stable formula $\tau(G) \sum_{\gamma_0} SO_{\gamma_0}(f^{\infty p} (\chi_c^{G(\qp)} f_\alpha) f_\infty)$.

By the argument at \cite[Lemma 4.1]{MR1163241} the above stable formula is the geometric side of the trace formula for the group $G$ and the function $\chi_c^{G(\qp)} f$;  therefore it is equal to the trace of $\chi_c^{G(\qp)} f$ on the space of automorphic forms $\cA(G)$ on $G$.  We have obtained that $T_B(f^{p}, \alpha)$ equals $\Tr(\chi_c^{G(\qp)} f, \cA(G))$ for all $\alpha \geq \alpha_0$. 
\end{proof}

\begin{definition}\label{steinbergtypedef}
We call a smooth representation $\pi_p$ of $G(\qp)$ of \emph{Steinberg type} if the following two conditions hold: (1) For all $F^+$-places $\wp$ above $p$ we have 
$$
\pi_\wp = \begin{cases}
\St_{\Gl_{n}(F^+_{\wp})}\otimes \phi_\wp & \wp \in \Ram_p \cr
\textup{Generic unramified} & \wp \in \Unr_p
\end{cases}
$$
where $\phi_\wp$ is an unramified character. (2) The factor of similitude $\qp^\times$ of $G(\qp)$ acts through an unramified character on the space of $\pi_p$.  
\end{definition}

\begin{lemma}\label{propertyA}
Let $\pi$ be an automorphic representation of $G$. Then $\pi$ is one-dimensional if the component $\pi_{\wp}$ is one-dimensional for some $F^+$-place $\wp$ above $p$. 
\end{lemma}
\begin{proof} 
Assume $\pi_\wp$ is one-dimensional. By twisting $\pi$ with a character we may assume that $\pi_\wp$ is the trivial representation. Let $H \subset G(\Af)$ be a compact open subgroup such that $\pi^H \neq 0$. We embed $\pi$ in the space of automorphic forms on $G$. Then elements of $\pi$ are complex valued functions on $G(\A)$. The group $U \subset G$ is the unitary group of elements whose factor of similitude is trivial, and this group $U$ arises by restriction of scalars from a unitary group $U'$ over $F^+$. Let $SU$ be the derived group of $U'$. Then $SU$ is a simply connected algebraic group over $F^+$. We may restrict the automorphic representation $\pi$ of $G$ to obtain a representation of the group $SU(\A_{F^+})$ (which is reducible in general). Let $h \in \pi$ be an element, then $h$ is a complex valued function on $G(\A_{F^+})$ invariant under the groups $SU(F^+), H$ and also under the group $SU(F_\wp^+)$ because $\pi_\wp$ is the trivial representation. By strong approximation for the group $SU$ we see that $SU(\A_{F^+})$ acts trivially on $h \in \pi^H$. Thus $SU(\A_{F^+})$ acts trivially on the space $\pi$. Therefore $\pi$ is an Abelian automorphic representation of $G$ and thus one-dimensional.  
\end{proof}

\begin{proposition}\label{intpropB}
For all $\alpha \geq \alpha_0$ the trace $T_B(f^p, \alpha)$ is equal to
\begin{equation}\label{expression.AA}
\sum_{\bo{\pi \subset \cA(G)}{\dim(\pi) = 1, \pi_p=\textup{Unr}}} \Tr(\chi_c^{G(\qp)} f, \pi) + \sum_{\bo{\pi \subset \cA(G)}{\pi_p=\textup{ St. type}}} \Tr(\chi_c^{G(\qp)} f, \pi),
\end{equation}
where both sums range over the irreducible subspaces of $\cA(G)$. 
\end{proposition}
\begin{proof}
Fix throughout this proof an automorphic representation $\pi \subset \cA(G)$ of $G$ such that $\Tr(\chi_c^{G(\qp)} f, \pi) \neq 0$. We base change $\pi$ to an automorphic representation $BC(\pi)$ of the algebraic group $\cK^\times \times D^\times$. Here we are using that $D$ is a division algebra and therefore the second condition in Theorem A.3.1(b) of the Clozel-Labesse appendix in \cite{MR1695940} is satisfied (cf. \cite[\S VI.2]{MR1876802} and \cite{shinnote}). In turn we use the Jacquet-Langlands correspondence \cite{vigneras} (cf. \cite[\S VI.1]{MR1876802} and \cite{MR2390289}) to send $BC(\pi)$ to an automorphic representation $\Pi := JL(BC(\pi))$ of the $\Q$-group $G^+ = \Res_{\cK/\Q} \Gm \times \Res_{F, \Q} \Gl_{n,F}$. 

The transferred representation $\Pi$ is discrete and $\theta$-stable, which means that $\Pi$ is isomorphic to the representation $\Pi^\theta$ obtained from $\Pi$ by precomposition $G^+(\A) \to G^+(\A) \to \End_\C(\Pi)^\times$ with $\theta$. Because $\Pi$ is a subspace of the space of automorphic forms $\cA(G^+)$ it comes with a natural intertwining operator $A_\theta \colon \Pi \isomto \Pi^\theta$ induced from the action of $\theta$ on $\cA(G^+)$ (here we are using that multiplicity one is true for the discrete spectrum of $G^+$). The group $G^+(\qp)$ is isomorphic to $G(\qp) \times G(\qp)$ and the  representation $\Pi_p$ is isomorphic to $\pi_p \otimes \pi_p$. We have $\Tr(\chi_c^{G(\qp)} f_{\alpha}, \pi_p) \neq 0$. Therefore $\pi_p$ is semi-stable by Corollary \ref{chifpss}. For each $F^+$-prime $\wp$ the component $\pi_{\wp}$ is equal to a component $\Pi_{\wp'}$ for some (any) $F$-place $\wp'$ above $\wp$. As the representation $\Pi$ is a \emph{discrete} automorphic representation of the group $G^+(\A)$ the component $\pi_{\wp} = \Pi_{\wp'}$ is a \emph{semi-stable rigid} representation by the Moeglin-Waldspurger theorem (Theorem \ref{temperedargument}). 

We prove a lemma before finishing the proof of Proposition \ref{intpropB}.

\begin{lemma}\label{propertyB}
Assume that $\pi$ is infinite dimensional and that $\Tr(\chi_c^{G(\qp)} f_\alpha, \pi) \neq 0$. Then the transferred representation $\Pi$ is cuspidal.
\end{lemma}
\begin{proof}
We use the divisibility conditions on $n$ and $s_\wp$ to see that $\Pi$ is cuspidal: Because of these conditions, the Proposition \ref{vanishunlesssteinberg} implies that the component $\pi_\wp$ of $\pi$  at the prime $\wp$ is an unramified twist of either the trivial representation or of the Steinberg representation if $\wp$ lies in the set $\Ram_p$, \ie if the basic isocrystal is not \'etale at $\wp$. The trivial representation is not possible by the Lemma \ref{propertyA} and the assumption that $\pi$ is infinite dimensional. There is at least one $\wp$ such that $b_\wp$ is not \'etale (thus $\Ram_p \neq 0$), and therefore the discrete representation $\Pi$ is an unramified twist of the Steinberg representation at some finite $F^+$-place. By the Moeglin-Waldspurger classification of the discrete spectrum, the $G^+(\A)$-representation $\Pi$ must be cuspidal. 
\end{proof}

\noindent \textit{Continuation of the proof of Proposition \ref{intpropB}}. If the prime $\wp \in \Unr_p$ is such that the basic isocrystal at $\wp$ \emph{is} \'etale at $\wp$ then the function $\chi_c^{\Gl_n(F^+_\wp)} f_{\wp}$ is simply the unit of the spherical Hecke algebra, hence unramified, and therefore $\pi_{\wp}$ is an unramified representation; because $\pi_{\wp}$ occurs in a cuspidal automorphic representation of $G^+(\A)$ the representation $\pi_{\wp}$ is furthermore generic by the result of Shalika \cite{MR348047}. By Lemmas \ref{propertyA} and \ref{propertyB} there are the following possibilities for $\pi$. Either $\pi$ is one-dimensional and the component $\pi_p$ is unramified, or $\pi$ is infinite dimensional, and the component $\pi_p$ is of Steinberg type. We have proved that $T_B(f^p, \alpha)$ is equal to
\begin{equation}
\sum_{\bo{\pi \subset \cA(G)}{\dim(\pi) = 1}} \Tr(\chi_c^{G(\qp)} f, \pi) + \sum_{\bo{\pi \subset \cA(G)}{\pi_p=\textup{ St. type}}} \Tr(\chi_c^{G(\qp)} f, \pi),
\end{equation}
where both sums range over the irreducible subspaces of $\cA(G)$. \end{proof} 

The main theorem is now essentially established, we only need to expand the above sums slightly further using the calculations that we did in the first two sections.

We define a number $\zeta_{\pi_p} \in \C$ for the two types of representations at $p$ that occur in Equation \eqref{expression.AA}: those of Steinberg type and the one-dimensional, unramified representations.

\begin{definition}\label{zetapidef}
Assume that $\pi_p = \one(\phi_p)$ is unramified and one-dimensional. We define
\begin{equation}\label{zetapidefformula}
\zeta_{\pi_p} \bydef \phi_c(q) \prod_{\wp \in \Ram_p} \phi_\wp(q^{s_\wp}) \in \C^\times,
\end{equation}
where $\phi_c$ is the character by which the factor of similitude acts on the space of $\pi_p$. 
 Assume that $\pi_p$ is of Steinberg type. Then for all $\wp \in \Ram_p$ we have $\pi_\wp \cong \St_{\Gl_n(F^+_\wp)}(\phi_\wp)$ for some unramified character $\phi_\wp$ of $F_\wp^{+\times}$. Let $\phi_c$ be the character by which the factor of similitude of $G(\qp)$ acts on the space of $\pi_p$. We define $\zeta_{\pi_p}$ again by Equation \ref{zetapidefformula}. 
\end{definition}

\begin{definition}\label{theweightofzetapi}
Let $\pi$ be a $\xi$-cohomological automorphic representation of $G$. The center $Z$ of $G$ contains the torus $\G_{\textup{m}}$. We may precompose the central character $\omega_\pi$ of $\pi$ with the inclusion $\A^\times \subset Z(\A)$ to obtain a character $\A^\times \to \C^\times$.  
Let $w \in \Z$ be the unique integer such that the composition
\begin{equation}\label{weilqnumberdescription}
\A^\times \to \C^\times \overset{ |\cdot |} \to \R^\times_{>0}
\end{equation}
is the character $||\cdot ||^{w/2}$.  
\end{definition}

\begin{lemma}\label{weilqnumber}
Let $\pi$ be a $\xi$-cohomological automorphic representation of $G$ which is either unramified and one-dimensional, or of Steinberg type at $p$. Then $\zeta_\pi$ is a Weil-$q$-number of weight $w/n$.  
\end{lemma}
\begin{proof} Let $\phi_\wp$ be the character of $\Gl_n(F_\wp^+)$ as defined in Equation \ref{zetapidef}. Let $\omega_\pi$ be the central character of $\pi$. Then $\omega_{\pi, \wp} = \phi_\wp^n$ for all $\wp \in \Ram_p$, and at the factor of similitude of $G(\qp)$ we have $\omega_{\pi, c} = \phi_{c}^n$. Thus, the number $\zeta_{\pi_p}$ is an $n$-th root of the number 
\begin{equation}\label{etaformule}
\eta_{\pi_p} := \omega_c(q) \prod_{\wp \in \Ram_p} \omega_{\pi, \wp}(q^{s_\wp}) \in \C^\times. 
\end{equation}
Thus, to prove that $\zeta_\pi$ is Weil-$q$-number, it suffices to prove that $\eta_{\pi_p}$ is a Weil-$q$-number. The central character $\omega_\pi$ is a Gr\"ossencharakter of the center $Z \subset G$. The center $Z$ of $G$ is the set of elements $z \in F^\times \subset D^\times$ such that the norm of $z$ down to $F^{+ \times}$ lies in the subset $\Q^\times \subset F^{+\times}$. Because $\pi$ is $\xi_\C$ cohomological we have $\omega_{\pi, \infty} = \xi^{-1}_\C|_{Z(\R)}$. Let $U_Z \subset Z$ be the subtorus consisting of elements in $F^\times$ whose norm to $F^{+\times}$ is equal to $1$. We have an exact sequence $\mu_2 \injects U_Z \times \G_{\textup{m}, \Q} \surjects Z$ of algebraic groups over $\Q$, where the injection is the  embedding on the diagonal and the surjection is the multiplication map $\varphi \colon (u, x) \mapsto ux$. We may restrict the character $\omega_\pi$ of $Z(\A)$ to the group $U_Z(\A) \times \A^\times$ and we obtain in this manner a character $\omega_{\pi,1}$ of $U_Z(\A)$ and a character $\omega_{\pi,2}$ of $\A^\times$. 

The component at $p$ of the mapping $\varphi \colon U_Z(\A) \times \A^\times \to Z(\A)$ is the identity mapping
$$
U_Z(\qp) \times \qp^\times = F_{\qp}^{+\times} \times \Qp^\times \to F_{\qp}^{+\times} \times \Qp^\times = Z(\Qp). 
$$
Let $K_1 \times K_2 \subset U(\A) \times \A^\times$ be a compact open subgroup such that $\omega_{\pi,1}$ is $K_1$-spherical and such that $\omega_{\pi, 2}$ is $K_2$-spherical. The group $U(\Q) K_1 \backslash U(\A)$ is \emph{compact} and therefore the product 
$$
\prod_{\wp \in \Ram_p} \omega_{\pi, \wp}(q^{s_\wp}) \in \C^\times, 
$$
(cf. Equation \eqref{etaformule}) is a Weil-$q$-number of weight $0$. The group $\Q^\times K_2 \backslash \A^\times$ is non-compact and thus $\omega_{\pi, c}(q)$ is a Weil-$q$-number whose weight is $w$, where $w$ is defined in Definition \ref{theweightofzetapi}. This completes the proof.  
\end{proof}

\begin{definition}\label{polP}
We write $P(q^\alpha)$ for the trace $\Tr_c(\chi_c^{G(\qp)} f_\alpha, \one)$. 
\end{definition}

In general $P(q^\alpha)$ is not a polynomial in $q^\alpha$, it depends on $\alpha$ in the following manner. The explicit description of the function $f_\alpha$ from Proposition \ref{kottwitzfunctionexplicit} shows 
\begin{equation}\label{explicitone}
P(q^\alpha) = \prod_{\wp \in \Ram_p} \Tr\lhk \chi_c^{\Gl_n(F_\wp^+)} \prod_{v \in V(\wp)} f_{n \alpha_v s_v}^{\Gl_n(F_\wp^+)} \ , \one \rhk.
\end{equation}
The traces in the product in Equation \eqref{explicitone} are computed in Subsection \ref{trivialrepresentation} (see Proposition \ref{cttrivial}). 

\begin{remark} 
In general the function $P(q^\alpha)$ is not a polynomial in $q^\alpha$. The number $c_{\wp, \alpha}$ depends on the class of $\alpha$ in the group $\Z/M\Z$, where $M$ is large such that the algebra $F^+ \otimes E_{\p, M}$ is isomorphic to a product of copies of $E_{\p, M}$. For the $\alpha$ that range over the elements of a fixed class $\li c \in \Z/M\Z$ there exists a polynomial $\Pol_{\li c} \in \C[X]$ such that $P(q^\alpha) = \left. \Pol_{\li c} \right|_{X = q^\alpha}$. 
\end{remark}

\begin{theorem}\label{maintheorem}
The trace of the correspondence $f^p \times \Phi_\p^\alpha$ acting on the alternating sum of the cohomology spaces
$\sum_{i=0}^\infty (-1)^i \uH_{\et}^i(B_{\lfq}, \iota^* \cL)$ is equal to 
\begin{equation}\label{introBBc}
P(q^\alpha) \lhk \sum_{\bo{\pi \subset \cA(G)}{\dim(\pi) = 1,\pi_p=\textup{unr}}}  \zeta_{\pi}^{\alpha} \cdot \Tr(f^p, \pi^p) + (-1)^{(n-1) \cdot \# \Ram_p} \sum_{\bo{\pi \subset \cA(G)}{\pi_p=\textup{ St. type}}} \zeta_{\pi}^{\alpha} \cdot \Tr(f^p, \pi^p)\rhk. 
\end{equation}
for all positive integers $\alpha$. 
\end{theorem}
\begin{proof} Assume that $\alpha \geq \alpha_0$. In Proposition \ref{intpropB} we established that 
$$
T_B(f^p, \alpha) = \sum_{\bo{\pi \subset \cA(G)}{\dim(\pi) = 1, \pi_p=\textup{Unr}}} \Tr(\chi_c^{G(\qp)} f, \pi) + \sum_{\bo{\pi \subset \cA(G)}{\pi_p=\textup{ St. type}}} \Tr(\chi_c^{G(\qp)} f, \pi). 
$$
Let $\pi$ be an automorphic representations which contributes to one of the above two sums.  We have
$$
\Tr(\chi_c^{G(\qp)} f, \pi) = \Tr(\chi_c^{G(\qp)} f_\alpha, \pi_p) \Tr(f^p, \pi^p). 
$$
For $\pi_p$ there are two possibilities: (1) $\pi_p$ is one-dimensional, (2) $\pi_p$ is of Steinberg type. In the first case we have
$$
\Tr(\chi_c^{G(\qp)} f_\alpha, \pi_p) = \zeta_\pi^\alpha \cdot P(q^\alpha). 
$$
In the second case we have
$$
\Tr(\chi_c^{G(\qp)} f_\alpha, \pi_p) = \zeta_\pi^\alpha \cdot \prod_{\wp | p} \Tr(\chi_c^{\Gl_n(F_\wp^+)} f_\wp, \St_{\Gl_n(F_\wp^+)}). 
$$
By Proposition \ref{cttrivial} we have 
$ \Tr(\chi_c^{\Gl_n(F_\wp^+)} f_\wp, \St_{\Gl_n(F_\wp^+)}) = (-1)^{n-1} \Tr(\chi_c^{\Gl_n(F_\wp^+)} f_\wp, \one)$ and therefore
$$
\Tr(\chi_c^{G(\qp)} f_\alpha, \pi_p) = \zeta_\pi^\alpha \cdot (-1)^{(n-1) \cdot \#\Ram_p} \cdot P(q^\alpha).
$$
Thus Equation \eqref{introBBc} is true for all $\alpha \geq \alpha_0$; observe that it must then be true for all $\alpha > 0$. This completes the proof. 
\end{proof}

\section{Applications}

In this section we deduce two applications from our main theorem. We first deduce an expression for the zeta function of the basic stratum in terms of the cohomology of a complex Shimura variety. In the second application we deduce an explicit formula for the dimension of the variety $B/\Fq$.  

\subsection{The number of points in $B$}\label{application.points} Let $I_p \subset K_p$ be the standard Iwahori subgroup at $p$. We use Theorem \ref{maintheorem} to deduce a formula for the zeta funcation of $B$ in terms of the cohomology of the complex variety $\Sh_{K^pI_p}(\C)$.  

\begin{corollary}\label{complexcohomologydimension}
We have 
\begin{align}\label{complexcoheq}
\# B(\fqa) &= | \Ker^1(G, \Q) | \cdot P(q^\alpha) \cdot \lhk \sum_{\one(\phi_p)} \sum_{i=0}^\infty (-1)^i \zeta_\pi^\alpha \dim \uH^{i}(\Sh_{K^p I_p}(\C), \cL)[\one(\phi_p)] \right. \cr
+ & \left. (-1)^{(n-1)\#\Ram_p}\sum_{\pi_p \textup{ St. type}} \sum_{i=0}^\infty (-1)^i \zeta_\pi^\alpha \dim\uH^i(\Sh_{K^p I_p}(\C), \cL)[\pi_p^{I_p}] \rhk
\end{align}
for all positive integers $\alpha$. The numbers $\zeta_\pi$ are roots of unity whose order is bounded by $n \cdot \# \lhk Z(\Q) \backslash Z(\Af) / (K \cap Z(\Af)) \rhk$.
\end{corollary}
\begin{proof}
Take $f^{\infty p} = \one_{K^p}$ and $\xi_\C$ the trivial representation of $G_\C$. By the Grothendieck-Lefschetz trace formula, the Theorem \ref{maintheorem} provides an expression for the cardinal $\# B(\fqa)$ for all positive integers $\alpha$:
\begin{align}\label{bpoints}
\# B(\fqa) = & P(q^\alpha) \lhk \sum_{\bo{\pi \subset \cA(G)}{\dim(\pi) = 1, \pi_\infty = \one}}  \zeta_{\phi}^{\alpha} \cdot \dim \lhk \pi^p \rhk^{K^p}
\rhk + \cr 
& + (-1)^{(n-1)\# \Ram_p} P(q^\alpha) \lhk \sum_{\bo{\pi \subset \cA(G)}{\pi_p=\textup{ St. type}}} \zeta_{\pi}^{\alpha} \cdot \ep(\pi_\infty) \dim \lhk \pi^p_{\textup{f}} \rhk^{K^p} \rhk, 
\end{align}
where $\ep(\pi_\infty)$ is the Euler-Poincar\'e characteristic $\sum_{i=0}^\infty (-1)^i \dim \uH^i(\ig, K_\infty; \pi_\infty)$. The representation $\xi$ at infinity is trivial; therefore the component at infinity of the central character of any  automorphic representation contributing to the sums in Equation \eqref{bpoints} is trivial as well. Thus the numbers $\zeta_{\pi}^{\alpha} \in \C^\times$ are roots of unity. The first part of the statement now follows from the formula of Matsushima \cite[Thm. VII.3.2]{MR1721403}. The bound on the order of the roots of unity $\zeta_{\pi}$ follows from the proof of Lemma \ref{weilqnumber}. 
\end{proof}

\begin{remark}
Note that $\# B(\fqa)$ is (for sufficiently divisible $\alpha$) a sum of powers of $q^\alpha$. This suggests that $B$ may have a decomposition in affine cells as in the case of signatures $(n-1, 1); (n, 0), \ldots (n,0)$. 
\end{remark}

\begin{example} 
Assume, on top of the running conditions from subsection 3.1, that $D$ is a quaternion algebra over its center $F$ and that this center is quadratic imaginary over $\Q$. In this case $E = \Q$ and $U(\R)$ is the standard unitary group $U(1, 1)$. The variety $\Sh_K$ is a smooth projective curve, and the basic stratum $B \subset \Sh_{K, \fp}$ is a finite \'etale variety. If we take $\alpha$ sufficiently divisible then we have $B(\lfp) = B(\fpa)$. After simplifying, Corollary \ref{complexcohomologydimension} states that the cardinal $\# B(\lfp)$ equals $2 \cdot \ep(\Sh_K(\C)) - \ep(\Sh_{K^pI_p}(\C))$, where $\ep(X)$ is the Euler-Poincar\'e characteristic of the complex curve $X$: 
$$
\ep(X) \bydef \dim \uH^0(X, \Q) - \dim \uH^1(X, \Q) + \dim \uH^2(X, \Q). 
$$ 
This suggests a result for $\Sh_K$ similar to the geometric description of Deligne and Rapoport \cite{MR0337993} of the reduction modulo $p$ of the modular curve $X_0(p)$.  
\end{example}

\subsection{A dimension formula}\label{application.dimension} In this subsection we show that the dimension of the variety $B/\Fq$ can be deduced from Corollary \ref{complexcohomologydimension}. The strategy is to look for the highest order terms in the combinatorial polynomials that describe the compact traces on the representations that occur in the cohomology of $B$.  

\begin{proposition}
The dimension of the variety $B/\Fq$ is equal to 
$$
\sum_{\wp \in \Ram_p} \lhk \sum_{v \in V(\wp)} \frac {s_v(1-s_v)}2 + \sum_{j = 0}^{s_\wp - 1} \lceil j \frac n{s_\wp} \rceil \rhk. 
$$
\end{proposition}
\begin{proof}
The Galois group $\Gal(\lfq/\fq)$ acts through a finite cyclic group on the set of geometric components of the variety $B/\Fq$. In particular the $\alpha$-th power of the Frobenius does not permute these components if $\alpha$ is sufficiently divisible, say divisible by $M \in \Z$ suffices. Assume from now on that $M$ divides $\alpha$. Then each irreducible component of the variety $B_{\F_{q^\alpha}}$ is a geometric component. Pick a component of maximal dimension and inside it a dense open affine subset. By Noether's normalization Lemma this affine subset is finite over an affine space $\A^d_{\F_{q^\alpha}}$ where $d$ is the dimension of $B$. Thus the number of $\F_{q^{\alpha}}$-points in $B$ is a certain constant times $q^{\alpha  d}$ plus lower order terms. From Equation \eqref{complexcohomologydimension} we obtain a formula of the form $\#B(\F_{q^{\alpha}}) = P(q^{\alpha}) \cdot C$ where $C$ is a complicated constant equal to a difference of dimensions of cohomology spaces. 

There are two ways to see that the constant $C$ is non-zero. First Fargues established in his thesis \cite{MR2074714} that the basic stratum is non-empty, and thus the constant $C$ is non-zero. Second, we sketch an argument for non-emptiness of $B$ using Theorem \ref{maintheorem}. Use an existence theorem of automorphic representations (for example \cite{MR818353}) to find after shrinking the group $K$ at least one automorphic representation $\pi$ of $G$ contributing to the sums in Theorem \ref{maintheorem}. By base change and Jacquet-Langlands we can send any such automorphic representation to an automorphic representation of the general linear group (plus similitude factor). By strong multiplicity one for $\Gl_n$ the contributing automorphic representations of $G$ are determined up to isomorphism by the set of local components outside any given finite set of places. Therefore we can find a Hecke operator $f^p$ which acts by $1$ on $\pi$ and by $0$ on all other automorphic representations contributing to Equation \eqref{bpoints} (which are \emph{finite} in number). The trace of the correspondence $f^{p} \times \Phi_\p^{r\alpha}$ acting on the cohomology of the variety $B$ is then certainly non-zero. In particular the variety has non-trivial cohomology and must be non-empty. Therefore the constant $C$ is non-zero. 

For the determination of the dimension we forget about the constant $C$. By increasing $M$ (and thus $\alpha$) if necessary we may (and do) assume that the $E_{\p, \alpha}$-algebra $F^+ \otimes E_{\p, \alpha}$ is split. Then, by Proposition \ref{kottwitzfunctionexplicit} we have 
$$
f_{\alpha} = \one_{q^{-\alpha} \zp^\times} \otimes \bigotimes_{\wp | p} \prod_{v \in V(\wp)} f_{n \alpha s_v}^{\Gl_n(F^+_\wp)} \in \cH_0(G(\qp)). 
$$

We make the formulas for the compact trace of $f_\alpha$ on the trivial representation and on the Steinberg representation explicit. Fix a $\wp$ and write  $f_\wp$ for the component of the function $f_\alpha$ at the prime $\wp$. Write $z := \# V(\wp)$. The Satake transform $\cS_T(\widehat \chi_{N_0} f_\wp^{(P_0)})$ is equal to the  polynomial
\begin{equation}\label{satakeoff}
q^{\alpha \sum_{v=1}^{z} \frac {s_v(n-s_v)}2} \sum_{(t_{1i}), (t_{2i}), \ldots, (t_{z i})} X_1^{\alpha (t_{11} + t_{21} + \ldots + t_{z 1})} X_2^{\alpha (t_{12} + t_{22} + \ldots + t_{z 2})} \cdots X_n^{\alpha (t_{1n} + t_{2n} + \ldots + t_{z n})} 
\end{equation}
in the ring $\C[X_1^{\pm 1}, X_2^{\pm 1}, \ldots, X_n^{\pm 1}]$. In the above sum, for an index $v$ given, the symbol $(t_{vi})$ ranges over the extended compositions of the number $s_v$ of length $n$ with the following properties: 
\begin{enumerate}
\item[(C1)] for each  index $i$ we have $t_{vi} \in \{0, 1\}$;
\item[(C1)] define for each $i$ the number $t_i$ to be the sum $t_{1i} + t_{2i} + \ldots + t_{z i}$, 
then we have 
\begin{equation}\label{explicitobtuseAA}
t_1 + t_2 + \cdots + t_i > \frac {s_\wp} n i,
\end{equation}
for every index $i \in \{1, 2, \ldots, n-1\}$. 
\end{enumerate}

The highest order term of the polynomial $P(q^{r\alpha})$ corresponds to extended composition $(t_i)$ of $s$ defined by the equalities
$$
t_1 + t_2 + t_3 + \ldots + t_i = \left \lfloor i \frac {s} n \right \rfloor + 1
$$ 
for all $i < n$. This extended composition gives the monomial
$$
q^{\alpha \sum_{v \in V(\wp)} \frac {s_v(n-s_v)}2} X_1^{\alpha} X_{\lceil \frac n{s_\wp} \rceil}^{\alpha} X_{\lceil 2 \frac n{s_\wp} \rceil}^{\alpha} \cdots X_{\lceil (s_\wp - 1) \frac n {s_\wp} \rceil}^{\alpha}
$$
of the truncated Satake function $\cS_T(\widehat \chi_{N_0} f^{(P_0)}_\wp) \in \C[X_*(T_\wp)]$. We evaluate this monomial at the Hecke matrix of the $T_\wp$-representation $\delta_{P_0}^{1/2}$ to obtain 
$$
q^{\alpha  \lhk \sum_{v \in V(\wp)} \frac {s_v(n-s_v)}2 + \sum_{j=0}^{s_\wp - 1} \frac {2 \lceil j \frac n {s_\wp} \rceil + 1 -n}2 \rhk } \in \C[q^{\alpha }].
$$
By summing over all $\wp \in \Ram_p$ we see that the dimension of the variety $B$ is equal to 
$$
\sum_{\wp \in \Ram_p} \lhk \sum_{v \in V(\wp)} \frac {s_v(1-s_v)}2 + \sum_{j = 0}^{s_\wp - 1} \lceil j \frac n{s_\wp} \rceil \rhk. 
$$
This completes the proof. 
\end{proof}

\bibliographystyle{plain}
\bibliography{grotebib}

\begin{thebibliography}{10}

\bibitem{MR1007299}
J.~Arthur and L.~Clozel.
\newblock {\em Simple algebras, base change, and the advanced theory of the
  trace formula}, volume 120 of {\em Annals of Mathematics Studies}.
\newblock Princeton University Press, Princeton, NJ, 1989.

\bibitem{MR518111}
J.~G. Arthur.
\newblock A trace formula for reductive groups. {I}. {T}erms associated to
  classes in {$G({\bf Q})$}.
\newblock {\em Duke Math. J.}, 45(4):911--952, 1978.

\bibitem{MR1285969}
A.-M. Aubert.
\newblock Dualit\'e dans le groupe de {G}rothendieck de la cat\'egorie des
  repr\'esentations lisses de longueur finie d'un groupe r\'eductif
  {$p$}-adique.
\newblock {\em Trans. Amer. Math. Soc.}, 347(6):2179--2189, 1995.

\bibitem{MR2390289}
A.~I. Badulescu.
\newblock Global {J}acquet-{L}anglands correspondence, multiplicity one and
  classification of automorphic representations.
\newblock {\em Invent. Math.}, 172(2):383--438, 2008.
\newblock With an appendix by Neven Grbac.

\bibitem{MR1721403}
A.~Borel and N.~Wallach.
\newblock {\em Continuous cohomology, discrete subgroups, and representations
  of reductive groups}, volume~67 of {\em Mathematical Surveys and Monographs}.
\newblock American Mathematical Society, Providence, RI, second edition, 2000.

\bibitem{MR571057}
W.~Casselman.
\newblock The unramified principal series of p-adic groups. {I}. {T}he
  spherical function.
\newblock {\em Compositio Math.}, 40(3):387--406, 1980.

\bibitem{MR818353}
L.~Clozel.
\newblock On limit multiplicities of discrete series representations in spaces
  of automorphic forms.
\newblock {\em Invent. Math.}, 83(2):265--284, 1986.

\bibitem{MR986793}
L.~Clozel.
\newblock Orbital integrals on {$p$}-adic groups: a proof of the {H}owe
  conjecture.
\newblock {\em Ann. of Math. (2)}, 129(2):237--251, 1989.

\bibitem{MR1068388}
L.~Clozel.
\newblock The fundamental lemma for stable base change.
\newblock {\em Duke Math. J.}, 61(1):255--302, 1990.

\bibitem{MR794744}
L.~Clozel and P.~Delorme.
\newblock Pseudo-coefficients et cohomologie des groupes de {L}ie r\'eductifs
  r\'eels.
\newblock {\em C. R. Acad. Sci. Paris S\'er. I Math.}, 300(12):385--387, 1985.

\bibitem{harrisproject}
L.~Clozel, M.~Harris, J.-P. Labesse, and B.-C. Ngo.
\newblock {\em On the Stabilization of the Trace Formula}.
\newblock International Press. 2011.

\bibitem{MR771670}
P.~Deligne, D.~Kazhdan, and M.-F. Vigneras.
\newblock {\em Repr\'esentations des groupes r\'eductifs sur un corps local}.
\newblock Travaux en Cours. [Works in Progress]. Hermann, Paris, 1984.

\bibitem{MR0337993}
P.~Deligne and M.~Rapoport.
\newblock Les sch\'emas de modules de courbes elliptiques.
\newblock In {\em Modular functions of one variable, {II} ({P}roc. {I}nternat.
  {S}ummer {S}chool, {U}niv. {A}ntwerp, {A}ntwerp, 1972)}, pages 143--316.
  Lecture Notes in Math., Vol. 349. Springer, Berlin, 1973.

\bibitem{MR2074714}
L.~Fargues.
\newblock Cohomologie des espaces de modules de groupes {$p$}-divisibles et
  correspondances de {L}anglands locales.
\newblock {\em Ast\'erisque}, 1(291):1--199, 2004.
\newblock Vari{\'e}t{\'e}s de Shimura, espaces de Rapoport-Zink et
  correspondances de Langlands locales.

\bibitem{MR1431137}
K.~Fujiwara.
\newblock Rigid geometry, {L}efschetz-{V}erdier trace formula and {D}eligne's
  conjecture.
\newblock {\em Invent. Math.}, 127(3):489--533, 1997.

\bibitem{MR1254737}
T.~C. Hales.
\newblock Unipotent representations and unipotent classes in {${\rm SL}(n)$}.
\newblock {\em Amer. J. Math.}, 115(6):1347--1383, 1993.

\bibitem{MR0340486}
Harish-Chandra.
\newblock Harmonic analysis on reductive {$p$}-adic groups.
\newblock In {\em Harmonic analysis on homogeneous spaces ({P}roc. {S}ympos.
  {P}ure {M}ath., {V}ol. {XXVI}, {W}illiams {C}oll., {W}illiamstown, {M}ass.,
  1972)}, pages 167--192. Amer. Math. Soc., Providence, R.I., 1973.

\bibitem{MR1876802}
M.~Harris and R.~Taylor.
\newblock {\em The geometry and cohomology of some simple {S}himura varieties},
  volume 151 of {\em Annals of Mathematics Studies}.
\newblock Princeton University Press, Princeton, NJ, 2001.
\newblock With an appendix by Vladimir G. Berkovich.

\bibitem{MR401654}
H.~Jacquet and R.~P. Langlands.
\newblock {\em Automorphic forms on {${\rm GL}(2)$}}.
\newblock Lecture Notes in Mathematics, Vol. 114. Springer-Verlag, Berlin,
  1970.

\bibitem{MR564478}
R.~E. Kottwitz.
\newblock Orbital integrals on {${\rm GL}_{3}$}.
\newblock {\em Amer. J. Math.}, 102(2):327--384, 1980.

\bibitem{MR683003}
R.~E. Kottwitz.
\newblock Rational conjugacy classes in reductive groups.
\newblock {\em Duke Math. J.}, 49(4):785--806, 1982.

\bibitem{MR761308}
R.~E. Kottwitz.
\newblock Shimura varieties and twisted orbital integrals.
\newblock {\em Math. Ann.}, 269(3):287--300, 1984.

\bibitem{MR809866}
R.~E. Kottwitz.
\newblock Isocrystals with additional structure.
\newblock {\em Compositio Math.}, 56(2):201--220, 1985.

\bibitem{MR868140}
R.~E. Kottwitz.
\newblock Base change for unit elements of {H}ecke algebras.
\newblock {\em Compositio Math.}, 60(2):237--250, 1986.

\bibitem{MR1044820}
R.~E. Kottwitz.
\newblock Shimura varieties and {$\lambda$}-adic representations.
\newblock In {\em Automorphic forms, {S}himura varieties, and {$L$}-functions,
  {V}ol.\ {I} ({A}nn {A}rbor, {MI}, 1988)}, volume~10 of {\em Perspect. Math.},
  pages 161--209. Academic Press, Boston, MA, 1990.

\bibitem{MR1163241}
R.~E. Kottwitz.
\newblock On the {$\lambda$}-adic representations associated to some simple
  {S}himura varieties.
\newblock {\em Invent. Math.}, 108(3):653--665, 1992.

\bibitem{MR1124982}
R.~E. Kottwitz.
\newblock Points on some {S}himura varieties over finite fields.
\newblock {\em J. Amer. Math. Soc.}, 5(2):373--444, 1992.

\bibitem{MR1485921}
R.~E. Kottwitz.
\newblock Isocrystals with additional structure. {II}.
\newblock {\em Compositio Math.}, 109(3):255--339, 1997.

\bibitem{MR1695940}
J.-P. Labesse.
\newblock Cohomologie, stabilisation et changement de base.
\newblock {\em Ast\'erisque}, 1(257):vi+161, 1999.
\newblock Appendix A by Clozel and Labesse, and Appendix B by Lawrence Breen.

\bibitem{MR1026752}
C.~M{\oe}glin and J.-L. Waldspurger.
\newblock Le spectre r\'esiduel de {${\rm GL}(n)$}.
\newblock {\em Ann. Sci. \'Ecole Norm. Sup. (4)}, 22(4):605--674, 1989.

\bibitem{MR2567740}
S.~Morel.
\newblock {\em On the cohomology of certain noncompact {S}himura varieties},
  volume 173 of {\em Annals of Mathematics Studies}.
\newblock Princeton University Press, Princeton, NJ, 2010.
\newblock With an appendix by Robert Kottwitz.

\bibitem{MR1265561}
D.~Ramakrishnan.
\newblock Pure motives and automorphic forms.
\newblock In {\em Motives ({S}eattle, {WA}, 1991)}, volume~55 of {\em Proc.
  Sympos. Pure Math.}, pages 411--446. Amer. Math. Soc., Providence, RI, 1994.

\bibitem{MR1411570}
M.~Rapoport and M.~Richartz.
\newblock On the classification and specialization of {$F$}-isocrystals with
  additional structure.
\newblock {\em Compositio Math.}, 103(2):153--181, 1996.

\bibitem{MR348047}
J.~A. Shalika.
\newblock The multiplicity one theorem for {${\rm GL}_{n}$}.
\newblock {\em Ann. of Math. (2)}, 100:171--193, 1974.

\bibitem{MR2800722}
S.~W. Shin.
\newblock Galois representations arising from some compact {S}himura varieties.
\newblock {\em Ann. of Math. (2)}, 173(3):1645--1741, 2011.

\bibitem{shinnote}
S.~W. Shin.
\newblock On the cohomological base change for unitary simultude groups.
\newblock {\em Preprint}, 2011.

\bibitem{MR1359141}
M.~Tadi{\'c}.
\newblock On characters of irreducible unitary representations of general
  linear groups.
\newblock {\em Abh. Math. Sem. Univ. Hamburg}, 65:341--363, 1995.

\bibitem{MR338277}
G.~van Dijk.
\newblock Computation of certain induced characters of p-adic groups.
\newblock {\em Math. Ann.}, 199:229--240, 1972.

\bibitem{vigneras}
M.-F. Vigneras.
\newblock On the global correspondence between gl$(n)$ and division algebras.
\newblock {\em Lecture notes from IAS}, 1984.

\bibitem{MR1989693}
J.-L. Waldspurger.
\newblock La formule de {P}lancherel pour les groupes {$p$}-adiques (d'apr\`es
  {H}arish-{C}handra).
\newblock {\em J. Inst. Math. Jussieu}, 2(2):235--333, 2003.

\bibitem{MR584084}
A.~V. Zelevinsky.
\newblock Induced representations of reductive p-adic groups. {II}. {O}n
  irreducible representations of {${\rm GL}(n)$}.
\newblock {\em Ann. Sci. \'Ecole Norm. Sup. (4)}, 13(2):165--210, 1980.

\end{thebibliography}

\bigskip

\noindent Arno Kret \\
\noindent Universit\'e Paris-Sud, UMR 8628, Math\'ematique, B\^atiment 425, \\
\noindent F-91405 Orsay Cedex, France

\end{document}